\documentclass[a4paper,11pt]{article}
\usepackage{amsmath,amsthm,amssymb,bbm,enumitem}
\usepackage[bookmarks=false]{hyperref}

\numberwithin{equation}{section}

{\theoremstyle{definition}\newtheorem{definition}{Definition}[section]

\newtheorem{remark}[definition]{Remark}

\newtheorem{conjecture}[definition]{Conjecture}}

\newtheorem{proposition}[definition]{Proposition}
\newtheorem{lemma}[definition]{Lemma}
\newtheorem{theorem}[definition]{Theorem}
\newtheorem{corollary}[definition]{Corollary}

\newlist{enumlist}{enumerate}{1}
\setlist[enumlist]{labelindent=0cm,label=\arabic*$^\circ$,labelwidth=2.5ex,labelsep=0.5ex,leftmargin=3ex,align=left,topsep=0.5ex,itemsep=1ex,parsep=1ex}

\newlist{itemlist}{itemize}{1}
\setlist[itemlist]{labelindent=0cm,label=$\bullet$,labelwidth=2.5ex,labelsep=0.5ex,leftmargin=3ex,align=left,topsep=0.5ex,itemsep=1ex,parsep=1ex}

\newcommand{\eps}{\varepsilon}

\newcommand{\R}{\mathbb{R}}
\newcommand{\C}{\mathbb{C}}
\newcommand{\sa}{_\text{sa}}
\newcommand{\so}{\text{\it so}}
\newcommand{\app}{\text{\it app}}
\newcommand{\cV}{\mathcal{V}}
\newcommand{\cW}{\mathcal{W}}
\newcommand{\N}{\mathbb{N}}
\newcommand{\recht}{\rightarrow}
\newcommand{\cZ}{\mathcal{Z}}
\newcommand{\Tr}{\operatorname{Tr}}
\newcommand{\vphi}{\varphi}
\newcommand{\cN}{\mathcal{N}}
\newcommand{\T}{\mathbb{T}}
\newcommand{\Z}{\mathbb{Z}}
\newcommand{\lspan}{\operatorname{span}}
\newcommand{\ovt}{\mathbin{\overline{\otimes}}}
\newcommand{\ot}{\otimes}
\newcommand{\bA}{\mathbf{A}}
\newcommand{\bM}{\mathbf{M}}
\newcommand{\bAtil}{\widetilde{\mathbf{A}}}
\newcommand{\bMtil}{\widetilde{\mathbf{M}}}
\newcommand{\actson}{\curvearrowright}
\newcommand{\cM}{\mathcal{M}}
\newcommand{\cA}{\mathcal{A}}
\newcommand{\cF}{\mathcal{F}}
\newcommand{\om}{\omega}

\newcommand{\ns}{\mathord{\text{\rm n}_{\text{\rm\tiny s}}}}

\newcommand{\id}{\mathord{\text{\rm id}}}
\newcommand{\rn}{\mathord{\text{\rm n}}}
\newcommand{\cB}{\mathcal{B}}
\newcommand{\cD}{\mathcal{D}}
\newcommand{\cE}{\mathcal{E}}
\newcommand{\si}{\sigma}
\newcommand{\bim}[3]{\mathord{\raisebox{-0.4ex}[0ex][0ex]{\scriptsize $#1$}{#2}\hspace{-0.25ex}\raisebox{-0.4ex}[0ex][0ex]{\scriptsize $#3$}}}

\begin{document}

\begin{center}
{\boldmath\LARGE\bf Paving over arbitrary MASAs \vspace{0.5ex}\\ in von Neumann algebras}

\bigskip

{\sc by Sorin Popa\footnote{Mathematics Department, UCLA, CA 90095-1555 (United States), popa@math.ucla.edu\\
Supported in part by NSF Grant DMS-1401718} and Stefaan Vaes\footnote{KU~Leuven, Department of Mathematics, Leuven (Belgium), stefaan.vaes@wis.kuleuven.be \\
    Supported by ERC Consolidator Grant 614195 from the European Research Council under the European Union's Seventh Framework Programme.}}

%
\end{center}

\begin{abstract}\noindent
We consider a paving property for a maximal abelian $^*$-subalgebra (MASA) $A$
in a von Neumann algebra $M$, that we call \emph{\so-paving}, involving approximation in the {\so}-topology, rather
than in norm (as in classical Kadison-Singer paving).
If $A$ is the range of a normal conditional expectation, then {\so}-paving is equivalent to
norm paving in the ultrapower inclusion $A^\omega\subset M^\omega$.
We conjecture that any MASA in any von Neumann algebra satisfies {\so}-paving.
We use \cite{MSS13} to check this for all MASAs in $\mathcal B(\ell^2\mathbb N)$, all Cartan subalgebras in amenable von Neumann
algebras and in group measure space II$_1$ factors arising from profinite actions.
By \cite{P13}, the conjecture also holds true for singular MASAs in II$_1$ factors, and we obtain here an improved paving size
$C\varepsilon^{-2}$, which we show to be sharp.
\end{abstract}

\section{Introduction}

A famous problem of R.V. Kadison and I.M. Singer in \cite{KS59} asked whether
the diagonal MASA (maximal abelian $^*$-subalgebra) $\mathcal D$ in the algebra $\mathcal B(\ell^2\mathbb N)$, of all linear bounded operators on
the Hilbert space $\ell^2\mathbb N$, satisfies the {\it paving property}, requiring that for any $x\in \mathcal B(\ell^2\mathbb N)$ with $0$ on the diagonal,
and any $\varepsilon > 0$, there exists a partition of $1$ with projections $p_1, ..., p_n \in \mathcal D$, such that $\|\sum_i p_i x p_i\|\leq \varepsilon \|x\|$.

In their striking recent work \cite{MSS13}, A. Marcus, D. Spielman and N. Srivastava
have settled this question in the affirmative, while also obtaining an estimate for the minimal number of projections necessary for such $\eps$-paving,
$\text{\rm n}(x,\eps) \leq 12^4\eps^{-4}$, $\forall x=x^*\in \mathcal B(\ell^2\mathbb N)$.

On the other hand, in \cite{P13} the paving property for $\mathcal D \subset \mathcal B(\ell^2\mathbb N)$ has been shown equivalent to the
paving property for the ultrapower inclusion $D^\omega \subset R^\omega$, where $R$ is the hyperfinite II$_1$ factor, $D$ is
its Cartan subalgebra and $\omega$ is a free ultrafilter on $\mathbb N$.
(Recall from \cite{D54}, \cite{FM77} that a subalgebra $A$ in a von Neumann algebra $M$ is a {\it Cartan subalgebra} if it is a MASA,
there exists a  normal conditional expectation of $M$ onto $A$ and the normalizer of $A$ in $M$, $\mathcal N_M(A)=\{u\in \mathcal U(M)\mid uAu^*=A\}$,
generates $M$.) It was also shown in \cite{P13} that if $A$ is a singular MASA in $R$, and more generally in an arbitrary  II$_1$ factor $M$,
then $A^\omega \subset M^\omega$ has the paving property, with corresponding paving size majorized by $C\eps^{-3}$. (Recall from \cite{D54} that a MASA $A\subset M$
is {\it singular} in $M$ if its normalizer is trivial, $\mathcal N_M(A)\subset A$.)

Inspired by these results, we consider in this paper a new, weaker, paving property for an arbitrary MASA $A$ in a von Neumann algebra $M$,
that we call \emph{\so-paving},  which requires that for any $x\in M\sa= \{x\in M \mid x=x^*\}$ and $\varepsilon>0$ there exists $n$ such that $x$ can be
$(\eps,n)$ \emph{\so-paved}, i.e., for any {\so}-neighborhood $\mathcal V$ of $0$
there exists a partition of $1$ with projections $p_1, ..., p_n$ in $A$ and an element $a\in A$ satisfying $\|a\|\leq \|x\|$ and $\|q(\sum_i p_ixp_i -a)q\| \leq
\varepsilon \|x\|$, for some projection $q\in M$ with $1-q\in \mathcal V$ (see Section 2). We prove that if
there exists a normal conditional expectation from $M$ onto $A$ then {\so}-paving is equivalent to the property that for any $x\in M\sa$ and $\varepsilon >0$
there exists $n$ such that $x$ can be  approximated in the {\so}-topology with elements that can be $(\eps,n)$ norm paved (see Theorem \ref{equiv.paving}).
If in addition $A$ is countably decomposable, then
{\so}-paving with uniform bound on the number $n$ necessary to $(\eps, n)$ {\so}-pave any $x\in M\sa$, is equivalent to
the ultrapower inclusion $A^\omega \subset M^\omega$ satisfying norm paving (with $M^\omega$ as defined in \cite{O85}).
In particular, this shows that {\so}-paving amounts to
norm paving in the case $\mathcal D \subset \mathcal B(\ell^2\mathbb N)$.

We conjecture that any MASA in any von Neumann algebra satisfies the {\so}-paving property (see \ref{conjecture}).
We use \cite{MSS13} to check this conjecture for all MASAs in $\mathcal B(\ell^2\mathbb N)$
(i.e., for the remaining case of  the diffuse MASA $L^\infty([0,1]) \subset \mathcal B(L^2([0,1]))$, see Section 3),
for all  Cartan subalgebras in amenable von Neumann algebras, as well as for any  Cartan subalgebra
in a group measure space II$_1$ factor arising from a free ergodic measure preserving profinite action (see Section 4). At the same time, we prove that for a von Neumann
algebra $M$ with separable predual, norm paving over a MASA $A\subset M$ occurs if and only if $M$ is of type I and there exists a
normal conditional expectation of $M$ onto $A$ (see \ref{thm.norm-paving}).

For singular MASAs $A\subset M$, where the conjecture already follows from results in \cite{P13},
we improve upon the paving size obtained there, by showing that any finite number of elements in $M^\omega$
can be simultaneously $\eps$-paved over $A^\omega$ with $n < 1 + 16 \eps^{-2}$ projections (see Theorem \ref{thm.singular-case}). Moreover, this estimate is sharp: given any MASA in a finite factor, $A\subset M$, and any $\eps>0$, there exists $x\in M\sa$ with zero expectation onto $A$, such that if $\|\sum_{i=1}^n p_i x p_i\| \leq \eps \|x\|$, for some partition of $1$ with
projections in $A$, then $n$ must be at least  $\eps^{-2}$ (see Proposition \ref{prop.lower-bound}). We include a discussion on the multi-paving size for $\mathcal D \subset \mathcal B(\ell^2\mathbb N)$ and more generally for Cartan subalgebras (see Remark \ref{rem.optimal}).

\section{A paving conjecture for MASAs}

We will consider several paving properties for a MASA $A$ in a von Neumann algebra $M$. For convenience we first recall
the  initial Kadison-Singer  paving property of \cite{KS59}, for which we use the following terminology.

\begin{definition}\label{def.paving}
An element $x \in M$ is said to be $(\eps,n)$ {\it pavable over} $A$
if there exist projections $p_1,\ldots,p_n \in A$ and
$a \in A$ such that $\|a\| \leq \|x\|$, $\sum_{i=1}^n p_i = 1$ and $\Bigl\| \sum_{i=1}^n p_i x p_i - a \Bigr\| \leq \eps \|x\|.$ We denote by $\text{\rm n}(A\subset M; x,\eps)$
(or just $\text{\rm n}(x,\eps)$, if no confusion is possible), the smallest such $n$.
Also, we say that $x$ is {\it pavable} (over $A$) if for every $\eps > 0$, there exists an $n$ such that $x$ is $(\eps,n)$ pavable. We say that $A\subset M$ has the {\it paving property} if any $x\in M$ is pavable. We will sometimes use the terminology {\it norm pavable/paving} instead of just pavable/paving,
when we need to underline the difference with other paving properties.
\end{definition}

It is not really crucial to impose $\|a\| \leq \|x\|$. Indeed, without that assumption, the element $a\in A$ in an $(\eps,n)$ norm paving of $x$
satisfies $\|a\|\leq (1+\varepsilon)\|x\|$ so that replacing $a$ by $a'=(1+\varepsilon)^{-1}a$, we have $\|a'\|\leq \|x\|$ and
$\|\sum_i p_i x p_i -a'\|\leq 2\varepsilon \|x\|$.

Also note that if there exists a normal conditional expectation $E$ of $M$ onto $A$, then the element $a \in A$ in an $(\eps,n)$ norm paving of $x$ satisfies $\|E(x) - a\| \leq \eps \|x\|$, so that $\|\sum_i p_i x p_i - E(x)\|\leq 2\varepsilon \|x\|$. In the presence of a normal conditional expectation, one often defines $(\eps,n)$ norm pavability by requiring the partition
$p_1, ..., p_n \in A$ to satisfy $\|\sum_i p_i x p_i -E(x)\|\leq \varepsilon \|x\|$.

Finally note that if $y_1, y_2 \in M\sa$ are $(\eps, n)$ pavable, then $y_1+i y_2$ is $(2\eps, n^2)$ pavable. Thus, in order to obtain the paving property
for $A\subset M$, it is sufficient to check pavability of self-adjoint elements in $M$.

We next define two weaker notions of paving, involving approximation in the {\so}-topology rather than in norm.

\begin{definition}
An element $x \in M$ is  $(\eps,n)$ \emph{\so-pavable over $A$} if for every strong neighborhood $\cV$ of $0$ in $M$,
there exist projections $p_1,\ldots,p_n \in A$,  an element $a \in A$ and a projection $q\in M$ such that $\|a\|\leq \|x\|$,
$\sum_{i=1}^n p_i = 1$, $\|q(\sum_i p_i x p_i -a)q\| \leq \varepsilon \|x\|$
and $1-q \in \cV$. We denote by $\ns(x,\eps)$ the smallest such $n$.
An element $x\in M$ is \emph{\so-pavable over $A$} if for any
$\eps>0$, there exists $n$ such that $x$ is $(\eps,n)$ {\so}-pavable. We say that $A\subset M$ has the \emph{\so-paving} property if any $x\in M\sa$ is {\so}-pavable.
\end{definition}

It is easy to see that if $M$ is a finite von Neumann
algebra with a faithful normal trace $\tau$ and $x\in M\sa$, then $x$ is $(\eps,n)$ {\so}-pavable iff
given any $\delta>0$ there exist a partition of $1$ with  projections $p_1, ..., p_n\in A$ and $a\in A\sa$, $\|a\|\leq \|x\|$, such that
the spectral projection $q$ of $\sum_i p_i x p_i -a$ corresponding to $[-\eps \|x\|, \eps \|x\| ]$ satisfies $\tau(1-q)\leq \delta$.
As pointed out in \cite[Remark 2.4.1$^\circ$]{P13}, if $\omega$
is a free ultrafilter on $\N$,  then  $x\in M\sa$ has this latter property
if and only if, when viewed as an element in $M^\omega$, it is pavable over the ultrapower MASA $A^\omega$ of $M^\omega$.

\begin{definition}
An element $x \in M$ is  $(\eps,n; \kappa)$ \emph{\app-pavable over $A$}
if it can be approximated in the $so$-topology by a net
of $(\eps, n)$ pavable elements in $M$, bounded in norm by $\kappa \|x\|$.  An element $x\in M$ is \emph{\app-pavable over $A$}  if there
exists $\kappa_0$ such that for any
$\eps>0$, there exists $n$ such that $x$ is $(\eps,n; \kappa_0)$ {\app}-pavable. We say that $A\subset M$ has the \emph{\app-paving} property if any $x\in M\sa$
is {\app}-pavable.
\end{definition}

Obviously, norm paving implies {\so}- and {\app}-paving, with n$(x,\eps) \geq \ns (x,\eps)$, $\forall x$.
The next result shows that if a MASA is the range of a normal conditional
expectation then {\so}- and {\app}-pavability are in fact equivalent.

\begin{proposition} \label{prop.equiv-paving}
Let $M$ be a von Neumann algebra and  $A \subset M$ a MASA with the property
that there exists a normal conditional expectation $E:M \rightarrow A$. Let $x\in M\sa$,  $n \in \N$, $ \eps > 0$.
\begin{enumlist}
\item If $x$ is $(\eps,n;\kappa)$ app-pavable for some $\kappa\geq 1$, then
$x$ is $(2\kappa\eps',n)$ so-pavable for any $\varepsilon'> \varepsilon$.
\item If $x$ is $(\eps,n)$ so-pavable, then $x$ is $(\eps',n;3)$ app-pavable for any $\varepsilon'> \varepsilon$.
\end{enumlist}
\end{proposition}

\begin{proof} Proof of $1^\circ$. Let  $x_j \in M\sa$ with $\|x_j\| \leq \kappa \|x\|$ for all $j$ and such that $x_j$
is $(\eps,n)$ pavable for all $j$ and $x_j$ converges to $x$ in the {\so}-topology. Fix $\eps' > \eps$. We prove that $x$ is $(2\kappa\eps',n)$ {\so}-pavable, i.e that
given any {\so}-neighborhood $\mathcal V$  of $0$, there exist a partition of $1$ with projections $p_1, ..., p_n \in A$,
an element $a\in A$ and $q\in \mathcal P(M)$ such
that $1-q\in \mathcal V$ and  $\|q(\sum_i p_i xp_i - a)q\|\leq 2\kappa \varepsilon'\|x\|$.

Note that by changing if necessary the multiplicity of the representation of $M$ on the Hilbert space $\mathcal H$,
we may assume that the given neighborhood $\mathcal V$
is of the form $\mathcal V=\{x\in M\sa  \mid \|x\xi\| \leq \alpha \}$, for some unit vector $\xi \in \mathcal H$ and $\alpha>0$.

For every $j$, choose a partition of $1$ by projections $p_{j,1},\ldots,p_{j,n} \in A$ and an element $a_j \in A$ such that
$$\Bigl\| \sum_{i=1}^n p_{j,i} x_j p_{j,i} - a_j \Bigr\| \leq \eps \|x_j\| \leq \kappa \eps \|x\| \; .$$

Applying the conditional expectation $E$, it also follows that $\|E(x_j) - a_j\| \leq \kappa \eps \|x\|$. Therefore,
$$\Bigl\| \sum_{i=1}^n p_{j,i} (x_j - E(x_j)) p_{j,i} \Bigr\| \leq 2 \kappa \eps \|x\| \; .$$
Define the self-adjoint elements
$$T_j = \sum_{i=1}^n p_{j,i}(x-E(x)) p_{j,i} \quad\text{and}\quad S_j = \sum_{i=1}^n p_{j,i} (x_j - E(x_j)) p_{j,i} \; .$$
Let $\delta=2(\varepsilon'-\varepsilon)\kappa \|x\|$. Recall that the normal conditional expectation $E$ is automatically faithful because its support is a projection in $A' \cap M = A$ and thus equal to $1$. So, we can apply Lemma \ref{lem.uniform} and since $x_j \recht x$ strongly, we get that $T_j - S_j \recht 0$ strongly.
Thus, there exists $j$ large enough such that $|S_j-T_j| \in \delta \cV$, i.e. $\|(T_j-S_j)\xi\| < \alpha\delta$.

We claim that if we denote by $q$ the spectral projection of $|T_j-S_j|$ corresponding to the interval $[0, \delta]$,
then $1-q \in \cV$, i.e. $\|(1-q)\xi\| < \alpha$. Indeed, for if not, then $\|(1-q)\xi\| \geq \alpha$
and thus $\| |T_j-S_j| (1-q) \xi\| \geq \alpha \delta$, implying that

$$
\|(T_j-S_j)\xi\|\geq  \| |T_j-S_j| (1-q)\xi\| \geq \alpha\delta > \|(T_j-S_j)\xi\| \; ,$$
a contradiction.

On the other hand, $a=E(x)$ satisfies $\|a\|\leq \|x\|$ and we also have the estimates
$$
\|q(\sum_{i=1}^n p_{j,i} (x - E(x)) p_{j,i})q\| = \|qT_jq\|\leq \|q(T_j-S_j)q\|  + \|qS_jq\| \leq \delta + 2\kappa \varepsilon \|x\|  = 2\kappa \varepsilon'\|x\| \; .
$$
This finishes the proof of $1^\circ$.

Proof of $2^\circ$. Note that if $\varepsilon'\geq 2$ then there is nothing to prove. So without any loss
of generality we may assume $0<\varepsilon < \varepsilon'<2$. Denote $\alpha=1-\frac{\eps' - \eps}{2}$
and $\gamma=1-\frac{\alpha\eps' - \eps}{6} $.   Note that $\varepsilon'< 2$
implies  $\alpha\varepsilon' > \varepsilon$, so $\gamma < 1$. We clearly also have $\gamma > \alpha$.

Let $x \in M\sa$ be $(\eps,n)$ \so-pavable. Fix an open {\so}-neighborhood $\cW$ of $0$ in $M$.
We construct an $(\eps',n)$-pavable element $y \in M\sa$ with $\|y\| \leq 3\|x\|$ and $x-y \in \cW$. We may assume that $x \neq 0$.

By the lower semicontinuity of the norm with respect to the {\so}-topology, it follows that the set
$$\cW_1 = \cW \cap \{h \in M \mid \|x-h\| > \gamma \|x\| \}$$
is an open {\so}-neighborhood  of $0$ in $M$. Choose an open {\so}-neighborhood $\cW_0$ of $0$ such that
 $\cW_0 + \cW_0 \subset \cW_1$.

Using Lemma \ref{lem.uniform} below to realize the second point, we can fix an \so-neighborhood $\cV_1$ of $0$ such that for every projection $q \in M$ with $1-q \in \cV_1$, we have that
\begin{itemlist}
\item $x - qxq \in \cW_0$~;
\item $qaq - a \in \cW_0$ for all $a \in A$ with $\|a\| \leq \|x\|$.
\end{itemlist}
Again using Lemma \ref{lem.uniform} below, we can fix an \so-neighborhood $\cV_0 \subset \cV_1$ of $0$ such that for every projection $q \in M$ with $1-q \in \cV_0$, we have the following property.
\begin{itemlist}
\item For any partition of $1$ with projections $p_1, ..., p_n \in A$, the spectral projection $q'$ of $\sum_i p_iqp_i$
corresponding to the interval $(1-(\frac{\alpha\eps' - \eps}{6n^2})^2, 1]$ satisfies $1-q'\in \cV_1$.
\end{itemlist}
Since $x$ is $(\eps,n)$ {\so}-pavable, we can choose projections $p_1,\ldots,p_n \in A$, an element $a \in A$ and a projection $q\in M$ such that
$\|a\| \leq \|x\|$, $\sum_{i=1}^n p_i = 1$,  $\|q(\sum_i p_i x p_i -a)q\|\leq \varepsilon \|x\|$ and $1-q\in \cV_0$.

For each $i$, let  $e_i$ be the spectral projection of $p_iqp_i$ corresponding to the interval $(1-(\frac{\alpha\eps' - \eps}{6n^2})^2, 1]$
and denote $q'=\sum_i e_i$. By the last of the above properties, we have $1-q' \in \cV_1$.
Define $y = q'(x-a)q' + a$ and note that $\|y\| \leq \|x-a\| + \|a\| \leq 3 \|x\|$. We will prove that $x-y \in \cW$ and that $y$ is $(\eps',n)$-pavable.

Indeed, because $1-q' \in \cV_1$, we have
$$x-y = (x-q'xq') + (q'aq'-a)  \in \cW_0 + \cW_0 \subset \cW_1 \; .$$
So, $x-y \in \cW$ and $\|y\|\geq \gamma \|x\|$. Since this implies $\|\gamma a\| \leq \|y\|$, in order to prove that $y$ is $(\eps',n)$-pavable, it is sufficient to prove that
$\|\sum_i p_i y p_i - \gamma a \| \leq  \eps' \|y\|$.
To see this, note first that we have
$$
\sum_i p_i y p_i -\gamma a = \sum_i p_iq'(x-a)q'p_i  + (1-\gamma)a
=\sum_i e_i (x-a) e_i  + (1-\gamma)a
$$
and thus
$$
\|\sum_i p_i y p_i -\gamma a\|\leq \|\sum_i e_i(x-a)e_i\| + (1-\gamma)\|x\|.
$$
Since by the definition of $e_i$, we have
$$\|e_i-e_iq\|^2=\|e_i - e_i q e_i\| = \|e_i - e_i(p_iqp_i)\|\leq (\frac{\alpha\eps' - \eps}{6n^2})^2 \; ,$$
it follows that $\|q'-q'q\|\leq \sum_i \|e_i-e_iq\|\leq n \frac{\alpha\eps' - \eps}{6n^2}
=\frac{\alpha\eps' - \eps}{6n}$.
Thus, since $e_i = q' p_i$, we get that
$$\|e_i-q'qp_i\| = \|(q' - q'q) p_i\| \leq \|q'q-q'\|\leq  \frac{\alpha\eps' - \eps}{6n} \; ,$$
implying that
\begin{align*}
\|\sum_i p_iyp_i - \gamma a\| & \leq \|\sum_i e_i(x-a)e_i\|  + (1-\gamma)\|x\| \\
& \leq \sum_i \|e_i-q'qp_i\| \, \|x-a\|+ \|q'q(\sum_i p_ixp_i-a)qq'\| + \\ &\hspace{4cm} \sum_i \|x-a\| \, \|e_i - p_iqq'\| +(1-\gamma)\|x\| \\
& \leq \frac{\alpha\varepsilon'-\varepsilon}{3} \|x-a\| + \varepsilon \|x\| +(1-\gamma) \|x\|
\leq \frac{5\alpha\varepsilon'+ \varepsilon}{6} \|x\| \\
&\leq \frac{5\alpha\varepsilon'+ \varepsilon}{6} \gamma^{-1}  \|y\| \leq \alpha \gamma^{-1}\varepsilon' \|y\| < \varepsilon'\|y\| \; ,
\end{align*}
with the two last inequalities holding true because $\eps < \alpha \eps'$ and $\alpha \gamma^{-1} < 1$.
\end{proof}

In the proof of the above Proposition \ref{prop.equiv-paving}, we used the following elementary lemma.

\begin{lemma}\label{lem.uniform}
Let $M \subset \cB(H)$ be a von Neumann algebra and $P \subset M$ a von Neumann subalgebra. Assume that $P$ is finite and that $E : M \recht P$ is a normal faithful conditional expectation. If $(x_k)$ is a bounded net in $M$ that strongly converges to $0$, then the nets $(x_k a)$ converge strongly to $0$ uniformly over all $a \in (P)_1$~:
$$\text{for every $\xi \in H$, we have that}\quad \lim_k\Bigl(\sup_{a \in (P)_1} \|x_k a \xi\|\Bigr) = 0 \; .$$
\end{lemma}
\begin{proof}
Since $P$ is finite, we can fix a normal semifinite faithful (nsf) trace $\Tr$ on $P$ with the property that the restriction of $\Tr$ to the center $\cZ(P)$ is still semifinite. Define the nsf weight $\vphi = \Tr \circ E$ on $M$ and the corresponding space $\cN_\vphi = \{x \in M \mid \vphi(x^* x) < \infty\}$. We complete $\cN_\vphi$ into a Hilbert space $H_\vphi$~: to every $x \in \cN_\vphi$ corresponds a vector $\hat{x} \in H_\vphi$ and $M$ is faithfully represented on $H_\vphi$ by $\pi_\vphi(x) \hat{y} = \widehat{xy}$.

Whenever $z \in \cZ(P)$ is a projection with $\Tr(z) < \infty$, we consider the normal positive functional $\vphi_z \in M_*$ given by $\vphi_z(x) = \vphi(z x z)$. Since these $\vphi_z$ form a faithful family of normal positive functionals on $M$, it suffices to prove that
\begin{equation}\label{eq.suff}
\lim_k \Bigl(\sup_{a \in (P)_1} \vphi_z(a^* x_k^* x_k a) \Bigr) = 0 \quad\text{for all projections}\;\; z \in \cZ(P) \;\;\text{with}\;\; \Tr(z) < \infty \; .
\end{equation}
We denote by $J_\vphi$ the modular conjugation on $H_\vphi$. Since $P$ belongs to the centralizer of the weight $\vphi$, we have that $\widehat{xa} = J_\vphi \pi_\vphi(a)^* J_\vphi \hat{x}$ for all $x \in \cN_\vphi$ and $a \in P$. For $z \in \cZ(P)$ with $\Tr(z) < \infty$ and $a \in P$, we then find that
$$\vphi_z(a^* x_k^* x_k a) = \|\widehat{x_k a z}\|^2 = \|J_\vphi \pi_\vphi(a)^* J_\vphi \widehat{x_k z}\|^2 \leq \|a\|^2 \, \vphi_z(x_k^* x_k) \; .$$
Since $\lim_k \vphi_z(x_k^* x_k) = 0$, we get \eqref{eq.suff} and the lemma is proved.
\end{proof}

\begin{remark}
For Lemma \ref{lem.uniform} to hold, both the finiteness of $P$ and the existence of the normal faithful conditional expectation $E : M \recht P$ are crucial. First note that the lemma fails for the diffuse MASA in $\cB(H)$. It suffices to take $M = \cB(L^2(\T))$ and $P = L^\infty(\T)$, w.r.t.\ the normalized Lebesgue measure on $\T$. Consider the unitary operators $a_n \in P$ given by $a_n(z) = z^n$. We can also consider the $(a_n)_{n \in \Z}$ as an orthonormal basis of $L^2(\T)$ and define $x_k$ as the orthogonal projection onto the closure of $\lspan\{a_n \mid n \geq k\}$. Then, $x_k \recht 0$ strongly. With $\xi(z) = 1$ for all $z \in \T$, we find that $\sup_n \|x_k a_n \xi\|_2 = 1$ for every $k$. So, the existence of the conditional expectation $E$ is essential.

The previous paragraph implies in particular that the lemma fails if $M = P = \cB(H)$. So also the finiteness of $P$ is essential.
\end{remark}

We will now relate {\so}- and {\app}-pavability properties for a MASA $A\subset M$ having a normal conditional expectation $E_A:M\rightarrow A$, with the
norm-pavability for the associated inclusion of ultrapower algebras $A^\omega \subset M^\omega$. We will only consider the case when $A$ is countably decomposable,
i.e.,  when there exists a normal faithful state $\varphi$ on $A$. We still denote by $\varphi$ its extension to $M$ given by $\varphi\circ E_A$.

For the reader's convenience, we recall Ocneanu's definition of the ultrapower of a von Neumann algebra, from \cite{O85}.
Thus,  given a free ultrafilter $\omega$ on $\mathbb N$, one lets $I_\omega$ be the C$^*$-algebra of all bounded sequences $(x_n)_n \in \ell^\infty(\mathbb N, M)$ that
are $s^*$-convergent to $0$ along the ultrafilter $\omega$. One denotes by $M^{0,\omega}$ the multiplier (also called the
bi-normalizer) of $I_\omega$ in $\ell^\infty(\mathbb N, M)$ (which is easily seen to be a C$^*$-algebra)
and one defines $M^\omega$ to be the quotient $M^{0,\omega}/I_\omega$. This is shown in \cite{O85} to be a von Neumann algebra, called the $\omega$-{\it ultrapower of} $M$.
Since the constant sequences are in the multiplier $M^{0,\omega}$, we have a natural embedding $M\subset M^\omega$. It is easy to see that if $M$ is an atomic von Neumann algebra,
then $M^\omega=M$, in particular  $\cB(\ell^2\mathbb N)^\omega = \mathcal B(\ell^2\mathbb N)$.

To define the ultrapower MASA $A^\omega\subset M^\omega$, one proceeds as in \cite[Section 1.3]{P95}. One lets $E^{0,\omega}_A:\ell^\infty(\mathbb N, M)\rightarrow \ell^\infty(\mathbb N, A)$
be the conditional expectation defined by $E^{0,\omega}_A((x_n)_n)=(E_A(x_n))_n$. One  notices that $E^{0,\omega}_A(I_\omega) = I_\omega \cap \ell^\infty(\mathbb N, A)=\{(a_n)
\in \ell^\infty(\mathbb N, A) \mid \lim_\omega \varphi(a_n^*a_n)=0\}$ and
that $\ell^\infty(\mathbb N, A) \subset M^{0,\omega}$. Finally, one defines
$A^\omega = (\ell^\infty(\mathbb N, A)+I_\omega)/I_\omega \simeq \ell^\infty(\mathbb N, A)/I_\omega \cap \ell^\infty(\mathbb N, A)$.
It follows that $A^\omega$ this way defined is a von Neumann subalgebra of $M^\omega$, with $E^{0,\omega}_A$
implementing a normal conditional expectation $E_{A^\omega}$, which sends the class of $(x_n)_n$ to the class of
$(E_A(x_n))_n$. Moreover, by \cite[Theorem A.1.2]{P95}, it follows that $A^\omega$ is a MASA in $M^\omega$. Note also that $E_{A^\omega}$
coincides with $E_A$ when restricted to constant sequences in $M\subset M^\omega$.
From the above remark, the ultrapower of $\mathcal D \subset \mathcal B(\ell^2\mathbb N)$ coincides with $\mathcal D \subset \mathcal B(\ell^2\mathbb N)$ itself.

\begin{theorem}\label{equiv.paving}
Let $M$ be a von Neumann algebra and $A\subset M$ a MASA with the
property that there exists a normal conditional expectation $E_A :M\rightarrow A$. Let $\omega$ be a free ultrafilter on $\mathbb N$ and denote
by $A^\omega \subset M^\omega$ the corresponding ultrapower inclusion.

\begin{enumlist}
\item An element $x\in M\sa$  is {\so}-pavable over $A$ if and only if $x$ is  {\app}-pavable over $A$. So, $A\subset M$ has the {\so}-paving property if and only if it has the {\app}-paving property.

\item Assume that $A$ is countably decomposable. Then $x \in M\sa$ is \so-pavable over $A$ if and only if $x$ is norm pavable over $A^\om$. More precisely,
if $x \in M\sa$ is $(\eps,n)$ \so-pavable, then $x$ is $(\eps,n)$ norm pavable over $A^\omega$~; conversely, if $x \in M\sa$ is $(\eps,n)$ norm pavable over $A^\omega$, then $x$ is $(\eps',n)$ \so-pavable for all $\eps' > \eps$.

\item Still assume that $A$ is countably decomposable. Then the uniform \so-paving property of $A \subset M$ is equivalent with the uniform paving property of $A^\om \subset M^\om$.
More precisely, if every $x \in M\sa$ is $(\eps,n)$ \so-pavable, then every $x \in M^\om\sa$ is $(\eps,n)$ norm pavable.
\end{enumlist}
\end{theorem}

\begin{proof}
$1^\circ$ follows immediately from Proposition \ref{prop.equiv-paving}.

To prove $2^\circ$ and $3^\circ$, we assume that $A$ is countably decomposable and it suffices to prove the following two statements for given $0 < \eps < \eps'$ and $n \in \N$.
\begin{itemlist}
\item If $x \in M^\omega\sa$ is represented by the sequence $(x_m) \in M^{0,\omega}$ of self-adjoint elements $x_m \in M\sa$ satisfying $\|x_m\| \leq \|x\|$ and if every $x_m$ is $(\eps,n)$ \so-pavable, then $x$ is $(\eps,n)$ norm pavable over $A^\omega$.

\item If $x \in M\sa$ is $(\eps,n)$ norm pavable over $A^\om$, then $x$ is $(\eps',n)$ \so-pavable.
\end{itemlist}

Since $A$ is countably decomposable, we can fix a normal faithful state $\varphi$ on $A$ and still denote
by $\varphi$ its extension $\varphi \circ E_A$ to $M$. Note that the $s^*$-topology on the unit ball of $M\sa$
coincides with the {\so}-topology, both being implemented by the norm $\| \, \cdot \, \|_\varphi$.

We start by proving the first of the two statements above. For every $m$, the self-adjoint element $x_m$ is $(\eps,n)$ \so-pavable. So we can take a partition of $1$ with projections $p^m_1, ..., p^m_n\in A$, a projection $q_m\in M$ and an element $a_m \in A$ such that $\|a_m\| \leq \|x_m\| \leq \|x\|$ and such that
$\|q_m(\sum_i p_i^m x p^m_i - a_m)q_m\|\leq \varepsilon \|x\|$ and $\varphi(1-q_m)\leq 2^{-m}$. Since $(x_m)$ and $\ell^\infty(\mathbb N, A)$ are both
contained in $M^{0,\omega}$, the sequences $((1-q_m)p_i^m(x_m-a_m)p_i^m)_m$
and $(p_i^m(x_m - a_m)p_i^m(1-q_m))_m$ belong to $I_\omega$.

Thus, if we denote $a = (a_m)$ and $p_i=(p^m_i)_m \in A^\omega$, $1\leq i \leq n$,
then $p_1, ..., p_n$   is a partition of $1$ with projections in $A^\omega$ and $p_i (x - a) p_i$
coincides with $(q_mp^m_i(x_m - a_m)p_i^mq_m)_m$ in $M^\omega$. It follows that $\sum_i p_i (x-a)p_i$
coincides with $(q_m \sum_i p_i^m (x_m - a_m)p^m_i q_m)_m$ in $M^\omega$ and thus has norm majorized by $\varepsilon \|x\|$. So we have proved that $x$ is $(\eps,n)$ norm pavable over $A^\omega$.

To prove the second of the two statements above, let $x\in M\sa$ be $(\varepsilon, n)$ norm pavable over $A^\omega$ (as an element in $M^\omega$). Let $\delta > 0$ be arbitrary. We have to prove that there exists an $a' \in A$ with $\|a'\| \leq \|x\|$, a partition of $1$ with projections $e_1, ..., e_n\in A$ and a projection $q\in M$
such that $\varphi(1-q)\leq \delta$ and $\|q\sum_i e_i (x-a')e_iq\|\leq \varepsilon' \|x\|$.

Take projections $p_1, ..., p_n \in A^\omega$ and $a \in A^\om\sa$ so that $\|a\| \leq \|x\|$ and such that
$\sum_i p_i = 1$, $\|\sum_i p_i x p_i - a \| \leq \varepsilon \|x\|$. Represent the $p_i$ by sequences $(p^m_i)_m$ with projections $p^m_i\in A$ such that $\sum_i p_i^m=1$ for all $m$, and represent $a$ by a sequence $(a_m)_m$ with $a_m \in A\sa$ and $\|a_m\| \leq \|a\|$ for all $m$.

We conclude that there exists a sequence of self-adjoint elements $(y_m)_m \in I_\omega$ of norm at most $3\|x\|$ such that the sequence
$(b_m)_m=(\sum_i p_i^m(x-a_m)p^m_i - y_m)_m$ satisfies $\|b_m\|\leq \varepsilon \|x\|$ for all $m$.
Since $(y_m)_m\in I_\omega$, we have $\lim_\omega \varphi(|y_m|)=0$, so that there exists a neighborhood $\mathcal V$ of $\omega$
such that the spectral projection $q_m$ of $|y_m|$ corresponding to $[0, (\varepsilon'-\varepsilon) \|x\|]$ satisfies $\varphi(1-q_m)\leq \delta$,
for any $m\in \mathcal V$. Thus, for any such $m$, if we let $a' = a_m$, $e_i=p^m_i$ and $q=q_m$, then we have
$$
\|q\sum_i e_i (x-a')e_iq\|\leq \|q_mb_mq_m\|+\|q_my_mq_m\|
\leq \varepsilon \|x\|+(\varepsilon'-\varepsilon) \|x\| \leq \varepsilon'\|x\| \; .
$$
\end{proof}


\begin{conjecture}\label{conjecture}
\begin{enumlist}
\item Any MASA in a von Neumann algebra, $A\subset M$, with the property that there exists a normal conditional expectation of $M$ onto $A$,
has the {\so}-paving property (equivalently the {\app}-paving property).
Also, while the equivalence between \so- and \app-pavability for an arbitrary MASA $A$ in a von Neumann algebra $M$ is still to be clarified,
any MASA $A\subset M$ (not necessarily the range of a normal expectation) ought to satisfy both these properties.

\item Going even further, we expect that the paving size satisfies the estimate $\ns(x, \eps) \leq C \eps^{-2}$, $\forall x\in M\sa$, for some universal constant $C>0$, independent of $A \subset M$.
\end{enumlist}
\end{conjecture}

\begin{remark}
$1^\circ$ There is much evidence for $1^\circ$ in the above conjecture. By \ref{equiv.paving}.3$^\circ$ and the fact that the ultrapower of
$\mathcal D \subset \mathcal B(\ell^2\mathbb N)$ coincides with $\mathcal D \subset \mathcal B(\ell^2\mathbb N)$, {\so}-pavability for this inclusion
is equivalent to
Kadison-Singer paving, proved to hold true by Marcus-Spielman-Srivastava in \cite{MSS13}. It was already noticed in \cite{P13}
that {\so}-pavability over the Cartan MASA of the hyperfinite II$_1$ factor, $D\subset R$,
is equivalent to pavability of $\mathcal D \subset \mathcal B(\ell^2\mathbb N)$, and thus holds true by \cite{MSS13}. In fact, more cases of the conjecture can be
deduced from \cite{MSS13}. Thus, we notice in Section 3 that any MASA in a
type I von Neumann algebra (such as a diffuse MASA in $\mathcal B(\ell^2\mathbb N)$) satisfy both {\so}- and {\app}-pavability.
Then in Section 4, we use \cite{MSS13} to prove that any Cartan MASA in an amenable von Neumann algebra, or in a
group measure space II$_1$ factor arising from a
free ergodic profinite action, has the {\so}-pavability property. On the other hand, the conjecture had  already been checked for singular MASAs in II$_1$ factors in \cite{P13}, and
Cyril Houdayer and Yusuke Isono pointed out that, modulo some obvious modifications, the proof in \cite{P13} works
as well for any singular MASA $A$ in an arbitrary von Neumann algebra $M$, once $A$ is the range of normal conditional expectation from $M$. Finally in Remark \ref{rem.other-examples}, we prove that \so-pavability also holds for a certain class of MASAs that are neither Cartan, nor singular.

$2^\circ$ The estimate on the paving size $\ns(x,\eps) \sim \eps^{-2}$, $\forall x\in M\sa$, in point $2^\circ$ of the conjecture is more speculative, and there is less evidence for it.
Based on results in \cite{P13}, we will show in Theorem \ref{thm.singular-case} that this estimate does hold true for singular MASAs. We will also show in Proposition \ref{prop.lower-bound} that this is the best
one can expect for the {\so}-paving size of any MASA in a II$_1$ factor and thus, since $\ns(D\subset R, \eps)=\text{\rm n}(\mathcal D \subset \mathcal B(\ell^2\mathbb N),\eps)$,
for the paving size in the Kadison-Singer problem as well (a fact already shown in \cite{CEKP07}). For the inclusions
$\mathcal D \subset \mathcal B(\ell^2\mathbb N)$, the order of magnitude of the $\eps$-pavings obtained in \cite{MSS13}
is $C\eps^{-4}$, but the techniques used there seem to allow obtaining the paving size $C\eps^{-2}$.
However, in order to prove Conjecture \ref{conjecture} in its full generality, in particular unifying the singular and the
Cartan MASA cases (including the diagonal inclusions $D_k \subset \mathcal B(\ell^2_k)$, $2\leq k \leq \infty$), which are quite different in nature,
a new idea may be needed.

$3^\circ$ The $(\eps, n)$ {\so}-paving in the case of a MASA $A\subset M$ with a normal conditional expectation $E_A:M \rightarrow A$ and a normal faithful state $\varphi$ on $M$
with $\varphi \circ E_A =\varphi$, should be compared with $(\eps,n)$ $L^2$-paving in the Hilbert norm $\| \, \cdot \, \|_\varphi$, which for $x\in M$, $E_A(x)=0$, requires
the existence of a partition of $1$ with projections $p_1, ..., p_n \in A$ such that $\|\sum_i p_i x p_i \|_\varphi \leq \eps \|x\|_\varphi$. This condition
is obviously weaker than {\so}-paving, with $\text{\rm n}(x, \eps) \geq \ns(x,\eps)$ bounded from below by
the $L^2$-paving size of $x$, $\forall x\in M\sa$. It was shown in \cite[Theorem 3.9]{P13} to always occur, with
paving size majorised by $\eps^{-2}$ (in fact the proof in \cite{P13} is for MASAs in II$_1$ factors,
but the same proof works in the general case; see also \cite[Theorem A.1.2]{P95} in this respect). The proof of Proposition \ref{prop.lower-bound} at the end of this paper shows that
the paving size is bounded from below by $\eps^{-2}$ for all MASAs in II$_1$ factors.
\end{remark}

\section{Paving over MASAs in type I von Neumann algebras}

Marcus, Spielman and Strivastava have proved in \cite{MSS13}  that for every self-adjoint matrix $T \in M_k(\C)$ with zeros on the diagonal and every $\eps > 0$, there exist $r$ projections $p_1,\ldots,p_r \in D_k(\C)$ with $r \leq (6/\eps)^4$, $\sum_{i=1}^r p_i = 1$ and $\|p_i T p_i \| \leq \eps \|T\|$ for all $i$ (see also \cite{Ta13,Va14} for alternative presentations of the proof). Thus, if $\mathcal D$ is the diagonal
MASA in $\mathcal B=\mathcal B(\ell^2\N)$, then $\mathcal D \subset \mathcal B$ has the paving property, with
$\rn(\mathcal D \subset \mathcal B; x, \varepsilon) \leq 12^4 \varepsilon^{-4}$, $\forall x=x^*\in \mathcal B$.

In this section, we deduce from this that any MASA $A$ in a type I von Neumann algebra $M$ has the {\it so}- and {\it app}-paving property.

We also prove that a MASA $A$ in a von Neumann algebra $M$ with separable predual has the norm paving property if and only if $M$ is of type I and there exists a normal conditional expectation of $M$ onto $A$.

We start by deducing the following lemma from \cite{MSS13}.

\begin{lemma}\label{lem.type-I}
Let $(X,\mu)$ be a standard probability space and $\cB = M_k(\C)$ or $\cB = \cB(\ell^2 \N)$ with the diagonal MASA $\cD \subset \cB$. Consider the unique normal conditional expectation $E$ of $\cB \ovt L^\infty(X)$ onto $\cD \ovt L^\infty(X)$. If $T \in \cB \ovt L^\infty(X)$ is a self-adjoint element with $E(T) = 0$ and if $\eps > 0$, there exist $r$ projections $p_1,\ldots,p_r \in \cD \ovt L^\infty(X)$ with $r \leq (6/\eps)^4$, $\sum_{i=1}^r p_i = 1$ and $\|p_i T p_i \| \leq \eps \|T\|$ for all $i$.
\end{lemma}

\begin{proof}
It suffices to consider $\cB = \cB(\ell^2 \N)$. Fix a self-adjoint $T \in \cB \ovt L^\infty(X)$ with $E(T) = 0$ and $\eps > 0$. Denote by $r$ the largest integer satisfying $r \leq (6/\eps)^4$. We represent $T$ as a Borel function $T : X \recht \cB$ satisfying $\|T(x)\| \leq \|T\|$ and $E(T(x))=0$ for all $x \in X$. Define $Y$ as the compact Polish space $Y := \{1,\ldots,r\}^\N$. For every $y \in Y$ and $i \in \{1,\ldots,r\}$, we denote by $p^y_i \in \cD$ the projection given by $p^y_i(k) = 1$ if $y(k) = i$ and $p^y_i(k) = 0$ if $y(k) \neq i$. Clearly, the projections $p^y_1,\ldots,p^y_r$ with $y \in Y$ describe precisely all partitions of $\cD$. Also, for every $i \in \{1,\ldots,r\}$, the map $y \mapsto p^y_i$ is strongly continuous.

Define the Borel map
$$\cV : Y \times X \recht [0,+\infty) : \cV(y,x) = \max_{i = 1,\ldots,r} \|p^y_i \, T(x) \, p^y_i\|$$
and the Borel set $Z \subset Y \times X$ given by $Z := \{(y,x) \in Y \times X \mid \cV(y,x) \leq \eps \|T\|\}$. For every $x \in X$, we have that $T(x) \in \cB$ with $\|T(x)\| \leq \|T\|$ and $E(T(x)) = 0$. So, by \cite{MSS13}, for every $x \in X$, there exists a $y \in Y$ such that $(y,x) \in Z$. Denoting $\pi : Y \times X \recht X : \pi(y,x) = x$, this means that $\pi(Z) = X$. By von Neumann's measurable selection theorem (see \cite{vN39} or \cite[Theorem 18.1]{Ke95}), we can take a Borel set $X_0 \subset X$ and a Borel function $F : X_0 \recht Y$ such that $\mu(X \setminus X_0) = 0$ and $(F(x),x) \in Z$ for all $x \in X_0$.

The Borel functions $p_i : X_0 \recht \cD : p_i(x) = p_i^{F(x)}$ then define a partition $p_1,\ldots,p_r$ of $\cD \ovt L^\infty(X)$ with the property that $\|p_i T p_i \| \leq \eps \|T\|$ for all $i$.
\end{proof}

\begin{proposition}
Let $M$ be a von Neumann algebra of type I with separable predual and $A \subset M$ an arbitrary MASA. Then $A \subset M$ has both the {\so}- and the {\app}-paving properties.

More precisely, for every $x \in M\sa$ and $\eps > 0$, we have that $\ns(x,\eps) \leq 12^4 \eps^{-4}$. Also, there exists a strongly dense $*$-subalgebra $M_0 \subset M$ with $A \subset M_0$ such that for every $x \in (M_0)\sa$ and $\eps > 0$, we have that $\rn(x,\eps) \leq 12^4 \eps^{-4}$.
\end{proposition}

\begin{proof}
Fix an arbitrary MASA $A \subset M$. There exist standard probability spaces $(X_k,\mu_k)_{k \in \N}$ and $(X_d,\mu_d)$, $(X_c,\mu_c)$ such that, writing $A_k = L^\infty(X_k)$ and similarly $A_d$, $A_c$, the MASA $A \subset M$ is isomorphic to a direct sum of MASAs of the form
\begin{align}
& D_k(\C) \ot A_k \subset M_k(\C) \ot A_k \; , \quad \ell^\infty(\N) \ovt A_d \subset \cB(\ell^2(\N)) \ovt A_d \quad\text{and}\label{eq.discrete-masa}\\
& L^\infty([0,1]) \ovt A_c \subset \cB(L^2([0,1])) \ovt A_c \; .\notag
\end{align}
For the first two of these MASAs, by Lemma \ref{lem.type-I}, we get that $\rn(x,\eps) \leq 12^4 \eps^{-4}$ for every self-adjoint element $x$.

For the rest of the proof, we consider $M = \cB(L^2([0,1])) \ovt L^\infty(X)$ and $A = L^\infty([0,1]) \ovt L^\infty(X)$ for some standard probability space $(X,\mu)$. Fix $x \in M\sa$ and $\eps > 0$. Let $n$ be the largest integer satisfying $n \leq 12^4 \eps^{-4}$. We prove that $x$ is $(\eps,n)$ {\so}-pavable. Choose an $\so$-neighborhood $\cV$ of $0$ in $M$. For every $r > 0$, denote by $q_r \in \cB(L^2([0,1]))$ the orthogonal projection onto the subspace $H_r \subset L^2([0,1])$ defined as
$$H_r = \{\xi \in L^2([0,1]) \mid \xi\;\;\text{is constant on every interval}\;\; [r^{-1}(i-1),r^{-1} i) \; , \; \forall i = 1,\ldots,r\} \; .$$
Define $\xi_{r,i} = \sqrt{r} \chi_{[r^{-1}(i-1),r^{-1} i)}$, so that $(\xi_{r,i})_{i=1,\ldots,r}$ is an orthonormal basis of $H_r$.

When $r \recht \infty$, we have that $q_r \recht 1$ strongly. So we can fix $r$ large enough such that $1 - (q_r \ot 1) \in \cV$. Denote by $e_i \in L^\infty([0,1])$ the projection given by $e_i = \chi_{[r^{-1}(i-1),r^{-1} i)}$. Define the vector functionals $\om_{ij} \in \cB(L^2([0,1]))_*$ given by $\om_{ij}(T) = \langle T \xi_{r,i},\xi_{r,j} \rangle$. Define $a \in A$ given by
$$a = \sum_{i=1}^r e_i \ot (\om_{ii} \ot \id)(x) \; .$$
By construction, $\|a\| \leq \|x\|$.

Define the isometry $V \in \cB(\C^r,L^2([0,1]))$ given by $V(\delta_i) = \xi_{r,i}$ for all $i=1,\ldots,r$. Define $y \in M_r(\C) \ot L^\infty(X)$ given by $y := (V^* \ot 1) x (V \ot 1)$. We also put $b = (V^* \ot 1)a (V \ot 1)$. Denoting by $E : M_r(\C) \ot L^\infty(X) \recht D_r(\C) \ot L^\infty(X)$ the natural conditional expectation, we have $E(y) = b$. By Lemma \ref{lem.type-I}, we thus find projections $f_1,\ldots,f_n \in D_r(\C) \ot L^\infty(X)$ such that $f_1 + \cdots + f_n = 1$ and
$\|f_k (y - b) f_k\| \leq \eps\|y\| \leq \eps \|x\|$ for all $k = 1,\ldots,n$.

Define the projections $a_{ki} \in L^\infty(X)$ such that $f_k = \sum_{i=1}^r E_{ii} \ot a_{ki}$. Then define the projections $p_k \in A$ given by $p_k = \sum_{i=1}^r e_i \ot a_{ki}$. By construction, we have
$$(V^* \ot 1) p_k x p_k (V \ot 1) = f_k y f_k \quad\text{for all}\;\; k=1,\ldots,n \; .$$
Therefore,
$$\Bigl\|(q_r \ot 1) \Bigl(\sum_{k=1}^n p_k x p_k - a\Bigr) (q_r \ot 1)\Bigr\| = \Bigl\| \sum_{k=1}^n (V^* \ot 1) p_k x p_k (V \ot 1) - b \Bigr\|
= \Bigl\|\sum_{k=1}^n f_k y f_k - b\Bigr\| \leq \eps\|x\| \; .$$
Since $1 - (q_r \ot 1) \in \cV$, we have shown that $x$ is $(\eps,n)$ {\so}-pavable.

For the final part of the proof, for notational convenience, we replace the interval $[0,1]$ by the circle $\T$. We define $M_0 \subset \cB(L^2(\T))$ as the $*$-algebra generated by $L^\infty(\T)$ and the periodic rotation unitaries. By construction, $M_0 \subset M$ is a dense $*$-subalgebra containing $A$. By Lemma \ref{lem.type-I}, every $x \in (M_0)\sa$ is $(\eps,12^4\eps^{-4})$-pavable for all $\eps > 0$.
\end{proof}

We finally prove that for a MASA $A$ in a von Neumann algebra $M$ with separable predual, the classical Kadison-Singer paving holds if and only if $M$ is of type I and $A$ is the range of a normal conditional expectation.

\begin{theorem}\label{thm.norm-paving}
Let $M$ be a von Neumann algebra with separable predual and $A \subset M$ a MASA. Then $A \subset M$ satisfies the norm paving property if and only if $M$ is of type I and $A$ is the range of a normal conditional expectation.

Also, unless $M$ is of type I and $A$ is the range of a normal conditional expectation, there exist singular conditional expectations of $M$ onto $A$.
\end{theorem}
\begin{proof}
If $M$ is of type I and $A$ is the range of a normal conditional expectation, then $A \subset M$ is isomorphic to a direct sum of the first two types of MASAs given by \eqref{eq.discrete-masa}. It then follows from Lemma \ref{lem.type-I} that $A \subset M$ satisfies the norm paving property.

Conversely, assume that $A \subset M$ satisfies the norm paving property. Then there is a unique conditional expectation $E : M \recht A$. By \cite[Corollary 3.3]{AS12}, this unique conditional expectation $E$ is normal.

Decomposing $M$ as a direct sum of von Neumann algebras of different types, it remains to prove the following: if $M$ has a separable predual and is of type II, type III$_1$ or type III without type III$_1$ direct summand and if $A \subset M$ is a MASA that is the range of a normal conditional expectation $E : M \recht A$, then there also exists a singular conditional expectation of $M$ onto $A$. When $M$ is of type II, the existence of a normal conditional expectation of $M$ onto $A$ implies that $A$ is generated by finite projections. By reducing with a projection in $A$, we may thus assume that $M$ is of type II$_1$, and in this case, singular conditional expectations were constructed in \cite[Remark 2.4.3$^\circ$]{P13} (see also in \cite[proof of Corollary 4.1.(iii) and Remark 4.3.3$^\circ$]{P97}).

To settle the type III cases, fix a normal faithful state $\vphi$ on $M$ satisfying $\vphi = \vphi \circ E$. First assume that $M$ is of type III$_1$ and fix $n \in \N$. We prove that there exist matrix units $\{e_{ij} \mid 1 \leq i,j \leq 2^n\}$ in $M$ such that $\|[\vphi,e_{ij}]\| \leq 8^{-n}$ for all $i,j$. To prove this statement, we use the following non-factorial version of the Connes-St{\o}rmer transitivity theorem \cite[Theorem 4]{CS76}: if $\vphi$ and $\rho$ are normal positive functionals on a type III$_1$ von Neumann algebra $M$ with separable predual and if $\vphi(a) = \rho(a)$ for all $a \in \cZ(M)$, then for every $\eps > 0$, there exists a unitary $u \in M$ such that $\|\vphi - u \rho u^*\| < \eps$.

Since $A$ is diffuse relative to $\cZ(M) \subset A$, we can choose a partition $e_{ii}$, $i=1,\ldots,2^n$, of $A$ satisfying $\vphi(a e_{ii}) = 2^{-n} \vphi(a)$ for all $a \in \cZ(M)$ and $i = 1,\ldots,2^n$. In particular, the projections $e_{ii}$ have central support $1$ and are thus equivalent in $M$. Put $v_1 = e_{11}$ and choose partial isometries $v_i$, $i=2,\ldots,2^n$ such that $v_i v_i^* = e_{11}$ and $v_i^* v_i = e_{ii}$ for all $i$. Define the positive functionals $\psi_i$ on $e_{11} M e_{11}$ given by $\psi_i(x) = \vphi(v_i^* x v_i)$. Whenever $z \in \cZ(e_{11} M e_{11})$, write $z = a e_{11}$ with $a \in \cZ(M)$, so that
$$\psi_i(z) = \vphi(v_i^* a v_i) = \vphi(a v_i^* v_i) = \vphi(a e_{ii}) = 2^{-n} \vphi(a) = \vphi(a e_{11}) = \psi_1(z) \; .$$
By the Connes-St{\o}rmer transitivity theorem, we can take unitaries $u_i \in e_{11} M e_{11}$ such that $\|\psi_1 - u_i \psi_i u_i^* \| \leq 8^{-n-1}$ for all $i$. Replacing $v_i$ by $u_i v_i$, this means that we may assume that $\|\psi_1 - \psi_i\| \leq 8^{-n-1}$ for all $i$. Define the matrix units $e_{ij} = v_i^* v_j$. Since $\vphi = \vphi \circ E$, we know that $[\vphi,e_{ii}] = 0$ for all $i$. We then find that $\|[\vphi,e_{ij}]\| \leq 8^{-n}$ for all $i,j$.

We now proceed as in \cite[Remark 2.4.3$^\circ$]{P13}. Define the projection $p_n = 2^{-n} \sum_{i,j} e_{ij}$. Since all $e_{ii}$ belong to $A$, we get that $E(e_{ij}) = \delta_{i,j} e_{ii}$ and thus $E(p_n) = 2^{-n} 1$. Since $\|[\vphi,e_{ij}]\| \leq 8^{-n}$ for all $i,j$, we also have $\|[\vphi,p_n]\| \leq 4^{-n}$. Define the normal states $\vphi_n$ on $M$ given by $\vphi_n(x) = 2^n \vphi(p_n x p_n)$, $x\in M$. Also define the normal functionals $\eta_n$ on $M$ given by $\eta_n(x) = 2^n \vphi(x p_n)$. Note that $\|\vphi_n - \eta_n\| \leq 2^{-n}$ and that $\eta_n(a) = \vphi(a)$ for all $a \in A$. So if $\psi$ denotes a weak$^*$ limit point of the sequence $\vphi_n$ in $M^*$, it follows that $\psi$ is a state on $M$ satisfying $\psi(a) = \vphi(a)$ for all $a \in A$. Defining the projection $q_n = \bigvee_{k=n+1}^\infty p_k$, we get that $\vphi(q_n) \leq 2^{-n}$ and thus $q_n \recht 0$ strongly. By construction, $\psi(1-q_n) = 0$ for every $n$. Therefore, $\psi$ is a singular state. Then averaging $\psi$ by a countable subgroup $\mathcal U_0\subset \mathcal U(A)$ with the property that $\mathcal  U_0''=A$, we get as in the proof of \cite[Corollary 4.1.(iii)]{P97} a singular state $\psi_0$ on $M$ that is $A$-central and whose restriction to $A$ equals $\vphi$. Then $\psi_0 = \vphi \circ \cE$ where $\cE : M \recht A$ is a singular conditional expectation (see e.g.\ \cite{K71}).

Finally assume that $M$ is of type III but without direct summand of type III$_1$. We prove that there exists an intermediate von Neumann algebra $A \subset P \subset M$ such that $P$ is of type II and $P$ is the range of a normal conditional expectation $M \recht P$.
(We are grateful to Masamichi Takesaki for useful
discussions on the discrete decomposition involved in this part of the proof.)
The first part of the proof  then shows the existence of singular conditional expectations $P \recht A$, which composed with the normal expectation of $M$ onto $P$  provides
singular conditional expectations $M \recht A$.

The intermediate type II von Neumann algebra $A \subset P \subset M$ can be constructed using the discrete decomposition for von Neumann algebras of type III$_\lambda$, $\lambda \in [0,1)$ (see \cite[Theorems XII.2.1 and XII.3.7]{Ta03}). To avoid the measure theoretic complications of a direct integral decomposition of $M$, we use the following ``global'' discrete decomposition. Denote by $(\si_t)_{t \in \R}$ the modular automorphism group of $\vphi$ and by $N = M \rtimes_{\si} \R$ the continuous core of $M$ (see \cite[Theorem XII.1.1]{Ta03}). Denote by $(\theta_t)_{t \in \R}$ the dual action of $\R$ on $N$. Write $\cZ(N) = L^\infty(Z,\mu)$ where $(Z,\mu)$ is a standard probability space. Note that $\theta$ restricts to a nonsingular action of $\R$ on $(Z,\mu)$. The assumption that $M$ has no direct summand of type III$_1$ is reflected by the possibility to choose $Z$ in such a way that no $x \in Z$ is stabilized by all $t \in \R$. This means that the flow $\R \actson (Z,\eta)$ can be built as a flow under a ceiling function (i.e.\ a non-ergodic version of \cite[Theorem XII.3.2]{Ta03}). More concretely, we find a nonsingular action of $\Z \times \R$ on a standard probability space $\Omega$ with the following properties.
\begin{itemize}
\item The actions of $\Z$ and $\R$ on $\Omega$ are separately free and proper, i.e.\ $\Z \actson \Omega$ is conjugate with $\Z \actson \Omega_0 \times \Z$ given by $n \cdot (x,m) = (x,n+m)$, and $\R \actson \Omega$ is conjugate with $\R \actson \Omega_1 \times \R$ given by $t \cdot (y,s) = (y,t + s)$.
\item The action $\R \actson Z$ is conjugate with the action $\R \actson \Omega / \Z$. So, we can identify $\Omega_0 = Z$ and thus $\Omega = Z \times \Z$ with the action $\R \actson \Omega$ given by $t \cdot (x,n) = (t \cdot x, \om(t,x) + n)$ where $\om : \R \times Z \recht \Z$ is a $1$-cocycle.
\end{itemize}
Since $L^\infty(Z) = \cZ(N)$, the $1$-cocycle $\om$ gives rise to a natural action $\R \actson N \ovt \ell^\infty(\Z)$. We define $\cN := (N \ovt \ell^\infty(\Z)) \rtimes \R$ and consider the action $\Z \actson \cN$ given by translation on $\ell^\infty(\Z)$ and the identity on $N$ and $L(\R)$. As in \cite[Lemma XII.3.5]{Ta03}, it follows that $\cN$ is of type II and that $\cN \rtimes \Z$ is naturally isomorphic with $M \ovt \cB(L^2(\R)) \ovt \cB(\ell^2(\Z))$.

Since $\vphi = \vphi \circ E$, we get that every $a \in A$ belongs to the centralizer of $\vphi$. We can then view $A \ovt L(\R)$ as a MASA of $N = M \rtimes_\si \R$. Also $\cZ(N) \subset A \ovt L(\R)$. So, the above action $\R \actson N \ovt \ell^\infty(\Z)$ globally preserves $A \ovt L(\R) \ovt \ell^\infty(\Z)$. We can then define $\cA := (A \ovt L(\R) \ovt \ell^\infty(\Z)) \rtimes \R$ as a von Neumann subalgebra of $\cN$.

The dual action $\R \actson L(\R)$ is conjugate with the translation action $\R \actson L^\infty(\R)$. Therefore, the $1$-cocycle $\om$ trivializes on $A \ovt L(\R)$. This yields the natural surjective $*$-isomorphism $$\Psi : A \ovt \cB(L^2(\R)) \ovt \ell^\infty(\Z) \recht \cA \; .$$ Choose a minimal projection $q \in \cB(L^2(\R)) \ovt \ell^\infty(\Z)$ and put $p = \Psi(1 \ot q)$.
We then get $A \subset p \cN p \subset p (\cN \rtimes \Z) p$. Using the natural isomorphism of $\cN \rtimes \Z$ with $M \ovt \cB(L^2(\R)) \ovt \cB(\ell^2(\Z))$, we can identify $p (\cN \rtimes \Z) p = M$ and have found $p \cN p$ as an intermediate type II von Neumann algebra sitting between $A$ and $M$. Because there is a natural normal conditional expectation of $\cN \rtimes \Z$ onto $\cN$, we also have a normal conditional expectation of $M$ onto $p \cN p$.
\end{proof}

\section{Paving over Cartan subalgebras}

The paving property for the diagonal MASA $\mathcal D\subset \mathcal B(\ell^2\mathbb N)$
was shown in \cite{P13} to be equivalent to the paving property for the ultrapower inclusion
$D^\omega \subset R^\omega$, where $D$ is the Cartan MASA in the hyperfinite II$_1$ factor $R$.
As we have seen in Theorem \ref{equiv.paving}, this is in turn equivalent to the (uniform) {\so}-paving property for $D\subset R$.
Thus, \cite{MSS13} implies that {\so}-paving holds true for $D\subset R$. We will now use \cite{MSS13} to prove that in fact {\so}-paving holds true
for any Cartan subalgebra of an amenable von Neumann algebra as well
as for Cartan inclusions arising from a free ergodic profinite probability measure preserving (pmp) action of a countable group, $\Gamma \curvearrowright X$,
i.e. $A=L^\infty(X)\subset L^\infty(X)\rtimes \Gamma=M$.

\begin{theorem}\label{thm.amenable-profinite}
\begin{enumlist}
\item If $M$ is an amenable von Neumann algebra and $A\subset M$ is a Cartan MASA of $M$, then $A\subset M$ has the so-paving property,
with $\ns(A\subset M; x, \varepsilon) \leq 25^4 \eps^{-4}$, $\forall x\in M\sa$.

\item Let $\Gamma$ be a countable group and $\Gamma \actson (X,\mu)$ an essentially free, ergodic, pmp action that is profinite. Then $A = L^\infty(X) \subset L^\infty(X) \rtimes \Gamma=M$
is {\so}-pavable and for every $x \in M\sa$, $\ns(A\subset M; x, \varepsilon) \leq 13^4 \eps^{-4}$.
So also, $A^\omega \subset M^\omega$ satisfies the norm paving property and for every $x \in M\sa^\omega$,
$\text{\rm n}(A^\omega \subset M^\omega; x, \varepsilon) \leq 13^4 \eps^{-4}$.
\end{enumlist}
\end{theorem}
\begin{proof} $1^\circ$ By \cite{CFW81}, given any $x\in M\sa$ and any {\so}-neighborhood $\mathcal V$ of $0$, there exists a finite dimensional von Neumann
subalgebra $B_0\subset M$, having the diagonal $A_0$ contained in $A$ and $\mathcal N_{B_0}(A_0)\subset \mathcal N_M(A)$, and an element $y_0=y_0^*\in B_0$,
$\|y_0\|\leq \|x\|$, such that $x-y_0\in \mathcal V$. But by \cite{MSS13} (see Lemma \ref{lem.type-I}), $y_0$ can be $(\eps_0,n)$ paved over $A_0$ (thus also over $A\supset A_0$), for some $\eps_0$ slightly smaller than $\eps/2$ and $n\leq 25^4\varepsilon^{-4}$. By Proposition \ref{prop.equiv-paving}, we conclude that $x$ can be $(\eps,n)$ {\so}-paved for every $\eps > 0$.

2$^\circ$ Take a decreasing sequence of finite index subgroups $\Gamma_n < \Gamma$ such that $(X,\mu)$ is the inverse limit of the spaces $\Gamma/\Gamma_n$ equipped with the normalized counting measure. Write $r_n : X \recht \Gamma/\Gamma_n$. The essential freeness of $\Gamma \actson (X,\mu)$ means that for every $g \in \Gamma - \{e\}$, we have
\begin{equation}\label{eq.trace}
\lim_n \frac{ \bigl| \{ x \in \Gamma/\Gamma_n \mid gx = x \}\bigr|}{[\Gamma : \Gamma_n]} = 0 \; .
\end{equation}
Write $A_n = \ell^\infty(\Gamma/\Gamma_n)$. View $A_1 \subset A_2 \subset \cdots$ as an increasing sequence of subalgebras of $A$ with dense union. Fix a free ultrafilter $\omega$ on $\N$. For every $n \in \N$, define $M_n \cong M_{[\Gamma,\Gamma_n]}(\C)$ as the matrix algebra with entries indexed by elements of $\Gamma/\Gamma_n$. Consider $A_n \subset M_n$ as the diagonal subalgebra. For $g \in \Gamma$, denote by $u_{g,n} \in M_n$ the corresponding permutation unitary. Denote by $\tau_n$ the normalized trace on $M_n$ and by $\|\,\cdot\,\|_2$ the corresponding $2$-norm. By \eqref{eq.trace}, we have that $\|E_{A_n}(u_{g,n})\|_2 \recht 0$ for all $g \in \Gamma - \{e\}$.

Denote by $\cM = \prod_\omega (M_n,\tau_n)$ the ultraproduct of the matrix algebras $M_n$, with MASA $\cA \subset \cM$ defined as $\cA = \prod_\omega A_n$. We can then define a normal faithful $*$-homomorphism $\pi : M \recht \cM$ where $\pi(a u_g) \in \cM$ is represented by the sequence $(a u_{g,n})_{n \geq m}$ whenever $a \in A_m$.

Fix $\eps > 0$ and denote by $r$ the largest integer that is smaller than or equal to $(12/\eps)^4$. We claim that for every self-adjoint $x \in M^\omega$, there exists a partition $p_1,\ldots,p_r$ of $A^\omega$ such that $\|p_i (x - E_{A^\om}(x)) p_i\| \leq \eps\|x\|$ for all $i$. To prove this claim, it suffices to prove the following local statement: for every self-adjoint $x \in M$ with $\|x\| \leq 1$ and for all $\delta > 0$, $m \in \N$, there exists a partition $p_1,\ldots,p_r$ of $A$ (thus, with $r$ fixed in the beginning, independent of $m$ and $\delta$) such that the element $y = \sum_{i=1}^r p_i (x - E_A(x)) p_i$ satisfies
\begin{equation}\label{eq.required-estimate}
|\tau(y^k)| \leq \eps^k + \delta \quad\text{for all}\;\; k=1,\ldots,m \; .
\end{equation}
Indeed, once this local statement is proved and given a self-adjoint element $x \in M^\om$ represented by a sequence $(x_m)_m$ with $x_m = x_m^*$ and $\|x_m\| \leq \|x\|$ for all $m$, we find partitions $p_1^m,\ldots,p_r^m$ of $A$ such that the elements $y_m = \sum_{i=1}^r p_i^m (x_m - E_A(x_m)) p_i^m$ satisfy
$$|\tau(y_m^k)| \leq (\eps \|x_m\|)^k + \frac{1}{m} \leq (\eps \|x\|)^k + \frac{1}{m} \quad\text{for all}\;\; k=1,\ldots,m \; .$$
Defining the projections $p_i \in A^\om$ by the sequences $p_i = (p_i^m)_m$ and putting $y = \sum_{i=1}^r p_i (x-E_{A^\om}(x)) p_i$, this means that $|\tau(y^k)| \leq (\eps \|x\|)^k$ for all $k \in \N$. Since $y$ is self-adjoint, it follows from the spectral radius formula that $\|y\| \leq \eps \|x\|$, so that the claim is proved. This means that every self-adjoint $x \in M^\omega$ can be $(\eps,n)$ paved for some $n \leq 12^4 \eps^{-4}$. So by Theorem \ref{equiv.paving}, also every $x \in M\sa$ can be $(\eps,n)$ \so-paved for some $n \leq 13^4 \eps^{-4}$.

We now deduce the above local statement from \cite{MSS13}. Fix $x \in M\sa$ with $\|x\| \leq 1$ and fix $\delta > 0$ and $m \in \N$. By the Kaplansky density theorem, we can take $n_0 \in \N$, a finite subset $\cF \subset \Gamma$ and a self-adjoint $x_0 \in \lspan \{ a u_g \mid a \in A_{n_0}, g \in \cF \}$ with $\|x_0\| \leq 1$ and $\|x-x_0\|_2 \leq \delta / (m 2^m)$. We may assume that $e \in \cF$. We prove below that we can find a partition $p_1,\ldots,p_r$ of $A$ such that the element $y_0 := \sum_{i=1}^r p_i (x_0 - E_A(x_0)) p_i$ satisfies $|\tau(y_0^k)| \leq \eps^k + \delta/2$ for all $k=1,\ldots,m$. Writing $y := \sum_{i=1}^r p_i (x-E_A(x)) p_i$, we find that $\|y-y_0\|_2 \leq \|x-x_0\|_2$ and also $\|y\| \leq 2$, $\|y_0\| \leq 2$. Therefore,
$$\|y^k - y_0^k\|_2 \leq m 2^{m-1} \|x-x_0\|_2 \leq \delta/2 \quad\text{for all}\;\; k = 1,\ldots, m \; . $$
Thus $|\tau(y^k) - \tau(y_0^k)| \leq \delta/2$, so that \eqref{eq.required-estimate} follows.

We now must find a good paving for $x_0$. For this, we use the ultraproduct $\cM$ and the injective homomorphism $\pi : M \recht \cM$ defined above. Write $x_0 = \sum_{g \in \cF} a_g u_g$ with $a_g \in A_{n_0}$. Then, $\pi(x_0)$ is represented by the bounded sequence of self-adjoint elements $T_n := \sum_{g \in \cF} a_g u_{g,n}$. Since $\|\pi(x_0)\| = \|x_0\| \leq 1$, we can take a bounded sequence of self-adjoint elements $S_n \in M_n$ such that $\lim_{n \recht \omega} \|S_n\|_2 = 0$ and $\|T_n - S_n\| \leq 1$ for all $n$. Take $K > 0$ such that $\|T_n\| \leq K$ and $\|S_n\| \leq K$ for all $n$. Take $n_1 \geq n_0$ close enough to $\omega$ such that $\|S_{n_1}\|_2 \leq \delta / (4 m (2K)^{m-1})$ and such that (using \eqref{eq.trace}) the projection $q \in A_{n_1}$ defined by the set
$$\{x \in \Gamma/\Gamma_{n_1} \mid \forall g \in \cF^m \setminus \{e\}, gx \neq x\}$$
satisfies $\|1-q\|_2 \leq \delta / 2^{m+2}$. Write $R = T_{n_1} - S_{n_1}$. Since $R = R^*$ and $\|R\| \leq 1$, by \cite{MSS13}, there exists a partition $p_1,\ldots,p_r$ of $A_{n_1}$ such that the element $Y := \sum_{i=1}^r p_i (R - E_{A_{n_1}}(R)) p_i$ satisfies
$$\|Y\| \leq \frac{\eps}{2} \| R - E_{A_{n_1}}(R) \| \leq \eps \; .$$
We also define $Z := \sum_{i=1}^r p_i (T_{n_1} - E_{A_{n_1}}(T_{n_1})) p_i$. Note that $\|Y\| \leq 2$ and $\|Z\| \leq 2K$. Also, $\|Y - Z\|_2 \leq \|S_{n_1}\|_2$, so that for all $k=1,\ldots,m$, we have
$$\|Y^k - Z^k\|_2 \leq m (2K)^{m-1} \|S_{n_1}\|_2 \leq \frac{\delta}{4} \; .$$
Then also $\|Y^k q - Z^k q \|_2 \leq \delta / 4$. Because $\|Y^k q\| \leq \|Y\|^k \leq \eps^k$, we conclude that
$$|\tau_{n_1}(Z^k q)| \leq \eps^k + \frac{\delta}{4} \quad\text{for all}\;\; k=1,\ldots,m \; .$$
By our choice of $q$, whenever $1 \leq k \leq m$, $a_1,\ldots,a_k \in A_{n_1}$ and $g_1,\ldots,g_k \in \cF$, we have
$$\tau_{n_1}(a_1 u_{g_1,n_1} \cdots a_k u_{g_k,n_k} \; q) = \tau(a_1 u_{g_1} \cdots a_k u_{g_k} \; q) \; ,$$
where the left hand side uses the trace in $M_{n_1}$, while the right hand side uses the trace in $M$. Writing $y_0 = \sum_{i=1}^r p_i (x_0 - E_A(x_0)) p_i$, we find that
$$|\tau(y_0^k q)| = |\tau_{n_1}(Z^k q)| \leq \eps^k + \frac{\delta}{4} \quad\text{for all}\;\; k = 1,\ldots,m \; .$$
Since $\|y_0^k q - y_0^k\|_2 \leq 2^m \|q-1\|_2 \leq \delta/4$, we get the required estimate
$$|\tau(y_0^k)| \leq \eps^k + \frac{\delta}{2} \quad\text{for all}\;\; k = 1,\ldots,m \; .$$
\end{proof}

\begin{remark}
We believe that \cite{MSS13} can be used to settle Conjecture \ref{conjecture} (i.e.\ {\so}-pavability) for all Cartan subalgebras in II$_1$ factors $A \subset M$, and in fact for any Cartan subalgebra  in a von Neumann algebra. The following could be an approach to a solution, but we could not make it work. Consider the abelian von Neumann algebra $\mathcal A=A \vee JAJ$ acting on $L^2(M)$. This is a MASA in $\mathcal M = \langle M, e_A \rangle=(JAJ)'\cap \cB(L^2(M))$ and there exists a normal conditional expectation
from the type I von Neumann algebra $\mathcal M$ onto $\mathcal A$
(see \cite{FM77}). Therefore, $\mathcal A \subset \mathcal M$ satisfies the norm-paving property. If now $x \in M$, we can pave $x$ by a partition $p_i \in A \vee JAJ$. Taking a very fine partition $q_j \in A$, we can {\so}-approximate $p_i$ by $\sum_j p_{i,j} J q_j J$. It should be possible to choose the $p_{i,j}$ as ``almost partitions'' of $1$ in $A$ such that for many $j$ (or at least one $j$), the $p_{1,j},\ldots,p_{r,j}$ approximately pave $x$ (in the {\so}-paving sense).
\end{remark}

In relation to the approach to proving {\so}-pavability for Cartan subalgebras suggested above, let us mention that the
\cite{MSS13} paving property for discrete MASAs in type I von Neumann algebras allows the following new characterization for a MASA to be Cartan.

\begin{corollary}
Let $M$ be a von Neumann algebra with separable predual and $A\subset M$ a MASA in $M$ that is the range of a normal conditional expectation. Denote $\mathcal M = \langle M, e_A \rangle=
(J A J)' \cap \mathcal B(L^2M)$ and $\mathcal A = A \vee JAJ$. The following conditions are equivalent.

\begin{enumlist}
\item\label{A-Cartan} $A$ is a Cartan subalgebra of $M$.


\item\label{cA-Cartan} $\mathcal A$ is a Cartan subalgebra of $\mathcal M$.

\item\label{cA-paving} $\mathcal A \subset \mathcal M$ has the paving property.
\end{enumlist}
\end{corollary}

\begin{proof}
The equivalence of \ref{A-Cartan} and \ref{cA-Cartan} follows from \cite{FM77}. Since $\mathcal M$ is of type I, a MASA in $\mathcal M$ is a Cartan subalgebra if and only if it is the range of a normal conditional expectation. Also, an abelian subalgebra of $\mathcal M$ can only satisfy the paving property if it is maximal abelian. Therefore, the equivalence of \ref{cA-Cartan} and \ref{cA-paving} follows from Theorem \ref{thm.norm-paving} (and thus, uses \cite{MSS13}).
\end{proof}

\section{Paving size for one or more elements}

In \cite{MSS13}, it is shown that every self-adjoint element $T $ in $\mathcal B(\ell^2_k)$, $1\leq k \leq \infty$, can be $(\eps,12^4 \eps^{-4})$-paved over its diagonal MASA.
In the previous section, we have used this result to prove that any amenable von Neumann algebra $M$ with a Cartan subalgebra $A\subset M$
is $(\varepsilon, 25^4\eps^{-4})$ {\so}-pavable over $A$, equivalently any self-adjoint element in $M^\omega$ is
$(\varepsilon, 25^4\eps^{-4})$ norm pavable over $A^\omega$.

On the other hand, it has been shown in \cite{P13} that
if $A$ is a singular MASA in a II$_1$ factor $M$, then n$(A^\omega\subset M^\omega; x, \varepsilon) \leq 25^2 \varepsilon^{-2}(\varepsilon^{-1}+1)\leq 1250 \varepsilon^{-3}$,
$\forall x\in M\sa^\omega$. Or equivalently, $\ns(A\subset M; x,  \varepsilon) \leq 1250 \varepsilon^{-3}$, $\forall x\in M\sa$ (see Corollary 4.3 and
last lines of the proof of Proposition 2.3 in \cite{P13}). This is shown by first proving that given any $\varepsilon>0$ and
any finite set of projections in $M$ that have scalar expectation onto $A$,
one can find a simultaneous {\so}-paving for all of them with at most $2\varepsilon^{-2}$ projections in $A$ (see \cite[Corollary 4.2]{P13}),
then using a dilation argument to deduce it for arbitrary selfadjoint elements.

We will now
show that in fact the {\so}-paving size for self-adjoint elements over singular MASAs, and respectively the norm-paving size
over an ultraproduct of singular MASAs,
can be improved to $4^2 \varepsilon^{-2}$ (N.B.: the order of magnitude $\varepsilon^{-2}$ for the paving size is optimal, see Proposition \ref{prop.lower-bound} below).
Moreover, we show that one can $(\eps,n)$ {\so}-pave simultaneously any number of selfadjoint
elements with $n< 1+ 4^2 \varepsilon^{-2}$ many projections over a singular MASA, a phenomenon that does not occur in the classical Kadison-Singer case
$\mathcal D \subset \mathcal B(\ell^2\mathbb N)$, nor in fact for any Cartan subalgebra in a II$_1$ factor $A\subset M$ (see Remark \ref{rem.optimal} below).
The proof combines the uniform paving of projections that have scalar expectation onto $A$ in \cite[Corollary 4.2]{P13} with
a better dilation argument that allows us not to lose on the paving size, while still dealing simultaneously with several self-adjoint elements.

\begin{theorem}\label{thm.singular-case}
Let $A_n \subset M_n$ be a sequence of singular MASAs in finite von Neumann algebras. Put $\bM = \prod_\omega M_n$ and $\bA = \prod_\omega A_n$.

Let $\eps > 0$. For every finite set of self-adjoint elements $\cF \subset \bM \ominus \bA$, there exists a decomposition of the identity $1 = p_1 + \cdots + p_n$ with $n < 1 + 16 \eps^{-2}$ projections $p_j \in \bA$ such that
$$\Bigl\| \sum_{j=1}^n p_j x p_j \Bigr\| \leq \eps \|x\| \quad\text{for all}\;\; x \in \cF \; .$$
\end{theorem}

\begin{proof}
Fix $\eps > 0$ and let $n$ be the unique integer satisfying $16 \eps^{-2} \leq n < 1+ 16\eps^{-2}$. Also fix a finite subset $\{x_1,\ldots,x_m\} \subset \bM \ominus \bA$ of self-adjoint elements. We may assume that $\|x_k\| = 1$ for all $k$.
Define $y_k = (1+x_k)/2$. Note that $0 \leq y_k \leq 1$ and $E_\bA(y_k) = 1/2$. Let $(B,\tau)$ be any diffuse abelian von Neumann algebra. Write
$$\bMtil = \prod_\omega (M_2(\C) \ot (M_n * B))$$
and consider the von Neumann subalgebra $\bAtil \subset \bMtil$ given by
$$\bAtil = \prod_\omega (A_n \oplus B) = \bA \oplus B^\omega \; .$$
Note that, for every $n$, we have that $A_n \oplus B \subset M_2(\C) \ot (M_n * B)$ is a singular MASA. Therefore, $\bAtil \subset \bMtil$ is the ultraproduct of a sequence of singular MASAs.

Define the orthogonal projections $Q_k \in \bMtil$ given by
$$Q_k = \begin{pmatrix} y_k & \sqrt{y_k - y_k^2} \\ \sqrt{y_k - y_k^2} & 1-y_k \end{pmatrix} \; .$$
Note that $E_{\bAtil}(Q_k) = 1/2$.

Applying \cite[Theorem 4.1.(a)]{P13} to $X = \{Q_k - 1/2 \mid k=1,\ldots,m\}$, we find a diffuse von Neumann subalgebra $B_0 \subset \bAtil$ such that every product with factors alternatingly from $B_0 \ominus \C 1$ and $X$ has zero expectation on $\bAtil$. In particular, for all $k$, we have that $B_0$ and $\C 1 + \C Q_k$ are free von Neumann subalgebras of $(\bMtil,\tau)$.

Choose any decomposition of the identity $1 = P_1 + \cdots + P_n$ with $n$ projections
$P_j \in B_0$ satisfying $\tau(P_j) = 1/n$. Fix $j \in \{1,\ldots,n\}$ and $k \in \{1,\ldots,m\}$.
Since the projections $P_j$ and $Q_k$ are free, with traces resp.\ given by $1/n$ and $1/2$, it follows from \cite[Example 2.8]{V86} that
$$\bigl\| P_j Q_k P_j - \frac{1}{2} P_j \bigr\| \leq \frac{2}{\sqrt{n}} \; .$$
Write $P_j = p_j \oplus q_j$ where $p_j \in \bA$ and $q_j \in B^\omega$ are projections. The upper left corner of $P_j Q_k P_j - \frac{1}{2} P_j$ equals $p_j \frac{x_k}{2} p_j$ and we conclude that
$$\|p_j x_k p_j \| \leq \frac{4}{\sqrt{n}} \leq \eps \; .$$
This ends the proof.
\end{proof}

\begin{remark}\label{rem.optimal} $1^\circ$ As shown in Theorem \ref{thm.singular-case} above,
in the case $A\subset M$ is singular, any finite number of elements can be simultaneously $(\eps,n)$ norm paved over $A^\omega$ with $n < 1 + 16\eps^{-2}$.
By \cite[Theorem 3.7]{P13}, any finite number of elements can also be simultaneously $(\eps,n)$ $L^2$-paved over $A^\omega$ with $n < 1 + \eps^{-2}$.
But this is no longer true for norm paving over a MASA that has ``large normalizer''.
For instance, one cannot pave multiple matrices in $\mathcal B(\ell^2\mathbb N)$ over its diagonal $\mathcal D$.
This can be seen as follows: assume $M$ is a finite von Neumann algebra and $A \subset M$ is a MASA whose normalizer $\cN_M(A)$
generates a II$_1$ von Neumann algebra. Thus, for any $m\geq 1$, there exists a unitary $u\in \cN_M(A)$ such that $E_A(u^k)=0$, $\forall 1\leq k \leq m-1$,
$u^m=1$.  Denote by $\sigma$ the automorphism $\text{\rm Ad}(u)$ of $A$.
Assume now that $p_1,\ldots,p_n$ is a partition of $A$ that simultaneously $c$-paves the set of $m-1$ unitaries $\{u^k\mid k = 1,\ldots,m-1\}$, for some $0<c<1$.
Then $\|p_i u^k p_i\| \leq c$ for all $i=1,\ldots,n$ and all $k = 1,\ldots,m-1$. But $\|p_i u^k p_i \| = \|p_i \sigma^k(p_i)\|$ and $p_i\sigma^k(p_i)$ is a projection.
Thus, $p_i\sigma^k(p_i)$ must be zero for all $i$ and $k$.
So, for every fixed $i$, we find that $p_i, \sigma(p_i),\ldots,\sigma^{m-1}(p_i)$ are orthogonal.
Thus, $\tau(p_i) \leq 1/m$. Since $\sum_i p_i = 1$, it follows that $n \geq m$. Note that by replacing the cyclic group $\mathbb Z/m\mathbb Z \simeq
\{u^k \mid 0\leq k \leq m-1\}\subset \mathcal N_M(A)$
with the group $(\mathbb Z/2\mathbb Z)^t \hookrightarrow \mathcal N_M(A)$, acting freely on $A$, one gets the same result for $m=2^t$, but with
a set of $m-1$ selfadjoint unitaries.

We conclude that if the normalizer of a MASA
generates a type II$_1$ von Neumann algebra, then given any $m$, there exists a set of $m-1$ unitaries in $M$ such that in order to
simultaneously $c$-pave all of them, with $c<1$, we need at least $m$ projections (in case $m=2^t$, the set can be taken
of self-adjoint unitaries). Note that,  if $u\in\mathcal N_M(A)$ is as before and we let
$X=\{(u^k + u^{-k})/2, (u^k- u^{-k})/2i \mid 1\leq k \leq m-1\}$, then any partition of $1$ with projections $p_1, ...,p_n \subset A$ that simultaneously $c/2$-paves
all $x\in X$, must satisfy $n\geq m=|X|/2+1$. Thus, under the same assumptions on $A\subset M$ as before, given any $m_0$ and any $c_0<1/2$, there exists a set
$X_0\subset M\sa$ with $|X_0|=m_0$ such that in order to simultaneously $c_0$-pave all $x\in X_0$, we need at least $m_0/2$ projections.

$2^\circ$ If $A \subset M$ is a MASA in a von Neumann algebra, $X\subset M$  and $\eps>0$, we define $\text{\rm n}(A\subset M; X, \eps)$ in the obvious way. Also,
for $m$ a positive integer, we let $\text{\rm n}(A\subset M; m, \eps)=\sup \{ \text{\rm n}(A\subset M; X, \eps)\mid X\subset M\sa, |X|=m \}$,
and call it the {\it multi-paving size} of $A\subset M$.
Note that one always has the estimate $\text{\rm n}(A\subset M; m, \eps) \leq \text{\rm n}(A\subset M; \eps)^m$.
By Theorem \ref{thm.singular-case}, if $A$ is a singular MASA in a II$_1$ factor $M$, then $\text{\rm n}(A^\omega\subset M^\omega; m, \eps) < 1+ 16 \eps^{-2}$, $\forall m\geq 1, \eps>0$.
By \ref{rem.optimal}.1$^\circ$ above, if $\mathcal N_M(A)''$ is of type II$_1$, then $\text{\rm n}(A\subset M; m-1, c)\geq m$, $\forall m=2^t$,
$0 < c <1$, while for arbitrary $m_0$ (not of the form $2^t$) and $c_0<1/2$, we have $\text{\rm n}(A\subset M; m_0, c_0)\geq m_0/2$. At the same time, by \cite{MSS13},
we have $\text{\rm n}(A\subset M; m, \eps)\leq (12/\eps)^{4m}$.

It would be interesting to find estimates for this multi-paving size in this last case (i.e., when $\mathcal N_M(A)$ is large). By arguing as in the proof of \cite[Theorem 2.2]{P13},
we see that $\text{\rm n}(\mathcal D\subset \mathcal B; m, \eps)=\text{\rm n}(D^\omega\subset R^\omega; m, \eps)=\text{\rm n}(\text{\bf D}\subset \text{\bf M}; m, \eps)$,
$\forall \eps>0, m\in \mathbb N$, where $\text{\bf D}\subset \text{\bf M}$ denotes the ultraproduct inclusion $\Pi_\omega D_k \subset \Pi_\omega M_{k \times k}(\mathbb C)$.
Thus, estimating the multi-paving size for $D^\omega \subset R^\omega$, or for $\text{\bf D}\subset \text{\bf M}$, is the same as doing it for $\mathcal D \subset \mathcal B$.
From \ref{rem.optimal}.1$^\circ$ and \cite{MSS13}, for each fixed $1>\eps>0$,  the growth in $m$ of the multiple paving size $\text{\rm n}(\mathcal D \subset \mathcal B;m, \eps)$ is
between $m$ and $(\eps^{-4})^m$. Calculating its order of magnitude seems a very challenging problem.
It would already be interesting to decide whether this growth is linear (more generally polynomial), or exponential.
\end{remark}

\begin{remark}\label{rem.other-examples}
Exactly the same proof as that of \cite[Theorem 4.1.(a)]{P13} shows
the following more general result. Let $(M,\tau)$ be a von Neumann algebra with a normal faithful tracial state, $A\subset M$ a MASA in $M$ and $A\subset N\subset M$
an intermediate von Neumann subalgebra with the following malnormality property: the only $A$-$N$-subbimodule of
$L^2(M \ominus N)$ that is finitely generated as a right $N$-module is $\{0\}$. Then, given any $\|  \cdot \|_2$-separable
subspace $X\subset M\ominus N$, and any free ultrafilter $\omega$ on $\N$, there exists a diffuse von Neumann subalgebra $B_0 \subset A^\omega$ such that every ``word''
with alternating ``letters''  from $B_0 \ominus \C 1$ and $X$ has trace zero. Note that \cite[Theorem 4.1.(a)]{P13} corresponds to the case $N=A$, because by \cite[Section 1.4]{P01},
the singularity of $A$ in $M$ implies that $L^2(M\ominus A)$ contains no non-zero $A$-$A$-subbimodule that is finitely generated as a right $A$-module.

By combining this result with the dilation argument as in the proof of Theorem \ref{thm.singular-case} above, it follows that any $x\in M\ominus N$ can be \so-paved, with $\ns(A\subset M; x, \eps) < 5^2\eps^{-2}$. Thus, if $A\subset N$ satisfies
the \so-paving property, then so does $A\subset M$, and we have the estimate $\ns(A\subset M; \eps) \leq  20^2 \eps^{-2} \, \ns(A\subset N; \eps/2)$.

This observation allows to derive the \so-paving property  (and thus the validity of \ref{conjecture}.1$^\circ$) for  a class
of MASAs that are neither singular nor Cartan. More precisely, assume that $A\subset M$ is a MASA in a II$_1$ factor such that  the normalizer $\cN_M(A)$ generates
a von Neumann algebra $N$ satisfying the conditions: $(1)$ either $N$ is amenable, or $A\subset N$ can be obtained as a group measure space construction from
a free ergodic profinite action of a countable group; $(2)$ $N\subset M$ satisfies the above malnormality condition. Then, $A\subset M$ has the \so-paving property.

Concrete such examples can be easily derived from \cite{P81}. For instance,  \cite[Theorem 5.1]{P81} provides an example
of a MASA $A$ in the hyperfine II$_1$ factor $M\simeq R$ such that the normalizer of $A$ in $M$ generates a subfactor $N\subset M$
with the property that $\bim{N}{L^2(M\ominus N)}{N}$ is an infinite multiple of the coarse $N$-$N$-bimodule $L^2(N)\otimes L^2(N)$, and thus $N\subset M$ satisfies the malnormality condition. Other examples come  from free product constructions: let $A \subset N$ be a Cartan subalgebra of a (separable) amenable von Neumann algebra of type II$_1$
(e.g., the hyperfinite II$_1$ factor, $N \simeq R$); let $(B,\tau)$ be a diffuse finite von Neumann algebra and denote $M = N * B$; then, $A$ is a MASA in $M$,
the normalizer of $A$ in $M$ generates $N$ and again, by \cite[Remark 6.3]{P81}, $\bim{N}{L^2(M\ominus N)}{N}$ is an infinite multiple of the coarse $N$-$N$-bimodule, so that $N\subset M$ satisfies the malnormality condition.
\end{remark}

We end with a result showing that the order of magnitude of the paving size obtained
in Theorem \ref{thm.singular-case} is optimal. More generally, we show that for any MASA in any II$_1$ factor the $\eps$-paving size is at least $\eps^{-2}$, i.e.,
$\sup \{\text{\rm n} (\eps, x) \mid x\in M\sa\}\geq \varepsilon^{-2}$. The proof is very similar to \cite[Theorem 6]{CEKP07}, where it was shown
that one needs at least $\eps^{-2}$ projections to $\eps$-pave self-adjoint unitary matrices.

\begin{proposition}\label{prop.lower-bound}
Let $M$ be a II$_1$ factor and $A \subset M$ a diffuse abelian von Neumann subalgebra.
Let $\eps > 0$ and $n < \eps^{-2}$. There exists a self-adjoint unitary $x \in M$ with $E_A(x)=0$ and
\begin{equation}\label{eq.lower}
\Bigl\| \sum_{k=1}^n p_k x p_k \Bigr\| \geq \Bigl\| \sum_{k=1}^n p_k x p_k \Bigr\|_2 > \eps
\end{equation}
for every decomposition of the identity $1=p_1 + \cdots + p_n$ with $n$ projections $p_k \in A$.

So if $A \subset M$ is a MASA in a II$_1$ factor, then the uniform $L^2$ paving size of $A^\om \subset M^\om$ is exactly equal to the smallest integer that is greater than or equal to $\eps^{-2}$.
\end{proposition}

\begin{proof}
Fix $\eps > 0$ and $n < \eps^{-2}$. Take $r$ large enough such that
\begin{equation}\label{eq.choice}
\frac{r}{r-1} \frac{1}{n} - \frac{1}{r -1} > \eps^2
\end{equation}
and such that there exists a conference matrix $C \in M_r(\R)$ of size $r$, i.e.\
$$C_{ij} = \pm 1 \;\;\text{if $i \neq j$,} \quad C_{ii} = 0 \;\;\text{for all $i$,} \;\;\text{and} \quad (r - 1)^{-1/2} C \;\;\text{is a self-adjoint unitary.}$$
Since $A$ is diffuse, we can choose projections $e_1,\ldots,e_r \in A$ with $1 = e_1 + \cdots + e_r$ and $\tau(e_i) = 1/r$ for every $i$. Since $M$ is a II$_1$ factor, we can choose partial isometries $v_1,\ldots,v_r \in M$ such that $v_i v_i^* = e_1$ and $v_i^* v_i = e_i$ for all $i$. Define
$$x = \frac{1}{\sqrt{r-1}} \sum_{i,j=1}^r C_{ij} v_i^* v_j \; .$$
Note that $x$ is a self-adjoint unitary. Since $A$ is abelian, we have for all $i \neq j$ that
$$0 = e_i e_j E_A(v_i^* v_j) = e_i E_A(v_i^* v_j) e_j = E_A(e_i v_i^* v_j e_j) = E_A(v_i^* v_j) \; .$$
Since $C_{ii} = 0$ for all $i$, we get that $E_A(x) = 0$.

Choose an arbitrary decomposition of the identity $1 = p_1 + \cdots + p_n$ with $n$ projections $p_k \in A$. We prove that \eqref{eq.lower} holds.
First note that
\begin{equation}\label{eq.first}
\Bigl\| \sum_{k=1}^n p_k x p_k \Bigr\|_2^2 = \sum_{k=1}^n \|p_k x p_k\|_2^2 = \sum_{k=1}^n \tau(p_k x p_k x) \; .
\end{equation}
Since $A$ is abelian, we can define the projections $p_{ik} = e_i p_k$. Writing $p_k = \sum_{i=1}^r p_{ik}$, we get for every $k \in \{1,\ldots,n\}$ that
\begin{align*}
\tau(p_k x p_k x) & = \sum_{i,j=1}^r \tau(p_{ik} x p_{jk} x) = \sum_{i,j=1}^r \tau(p_{ik} x p_{jk} x e_i) \\
&= \frac{1}{r-1} \sum_{i,j=1}^r C_{ij}^2 \tau( p_{ik} v_i^* v_j p_{jk} v_j^* v_i) \\
&= \frac{1}{r-1} \Bigl(\sum_{i,j=1}^r \tau( v_i p_{ik} v_i^* \; v_j p_{jk} v_j^*) - \sum_{i=1}^r \tau( v_i p_{ik} v_i^* \; v_i p_{ik} v_i^*)\Bigr) \\
&= \frac{1}{r-1} \bigl(\tau(T_k^2) - \tau(p_k)\bigr) \quad\text{where}\quad T_k = \sum_{i=1}^r v_i p_{ik} v_i^* \; .
\end{align*}
In combination with \eqref{eq.first}, it follows that
\begin{equation}\label{eq.second}
\Bigl\| \sum_{k=1}^n p_k x p_k \Bigr\|_2^2 = \frac{1}{r-1} \tau\Bigl( \sum_{k=1}^n T_k^2 \Bigr) - \frac{1}{r-1}\; .
\end{equation}
We next observe that, as positive operators, we have
\begin{equation}\label{eq.cs}
\sum_{k=1}^n T_k^2 \geq \frac{1}{n} \Bigl(\sum_{k=1}^n T_k \Bigr)^2 \; .
\end{equation}
Indeed, defining the elements $T,R \in M_{1,n}(\C) \ot M$ given by
$$T = (T_1 \; T_2 \; \cdots \; T_n) \quad\text{and}\quad R = (1 \; 1 \; \cdots \; 1) \; ,$$
we get that
$$\Bigl(\sum_{k=1}^n T_k \Bigr)^2 = T R^* R T^* \leq \|R\|^2 \, T T^* = n \sum_{k=1}^n T_k^2 \; .$$
So, \eqref{eq.cs} follows.  By construction, we have that $\sum_{k=1}^n T_k = r e_1$. So, in combination with \eqref{eq.second} and \eqref{eq.choice}, we find that
$$\Bigl\| \sum_{k=1}^n p_k x p_k \Bigr\|_2^2 \geq \frac{1}{r-1} \frac{1}{n} \tau(r^2 e_1) - \frac{1}{r-1} = \frac{1}{r-1} \, \frac{r}{n} - \frac{1}{r-1} > \eps^2 \; .$$
Thus we have proved \eqref{eq.lower}.

Now assume that $A \subset M$ is a MASA in the II$_1$ factor $M$. It follows that the uniform $L^2$ paving size of $A^\om \subset M^\om$ is at least $\eps^{-2}$. On the other hand, if $n$ is an integer and $n \geq \eps^{-2}$, it was proved in \cite[Section 3]{P13} that every element $x \in M^\om$ can be $(\eps,n)$ $L^2$-paved.
\end{proof}

\end{document}